\documentclass[reqno,12pt]{amsart}
\usepackage{bbm}
\usepackage[all,pdf]{xy}
\usepackage{epsfig}
\usepackage{amsmath}
\usepackage{amssymb}
\usepackage{amscd}
\usepackage{graphicx}
\usepackage[colorlinks=true]{hyperref}
\hypersetup{urlcolor=red, citecolor=blue}
\makeatletter
\@namedef{subjclassname@2020}{%
	\textup{2020} Mathematics Subject Classification}

\newtheorem{thm}{Theorem}[section]
\newtheorem{lem}[thm]{Lemma}

\newtheorem{cl}[thm]{Claim}
\newtheorem{prop}[thm]{Proposition}

\newtheorem{ques}[thm]{Question}
\newtheorem{cor}[thm]{Corollary}
\newtheorem{defn}[thm]{Definition}

\newtheorem{rem}[thm]{Remark}

\def\B {\mathcal B}
\numberwithin{equation}{section}

\begin{document}
	\title[Ergodic measures in minimal systems]{Ergodic measures in minimal group actions with finite topological sequence entropy}
	
	\author[Chunlin Liu, Xiangtong Wang and Leiye Xu]{Chunlin Liu, Xiangtong Wang and Leiye Xu}
	
	\address{C. Liu: CAS Wu Wen-Tsun Key Laboratory of Mathematics, University of Science and Technology of China, Hefei, Anhui, 230026, P.R. China}
	
	\email{lcl666@mail.ustc.edu.cn}	
	
	\address{X. Wang: CAS Wu Wen-Tsun Key Laboratory of Mathematics, University of Science and Technology of China, Hefei, Anhui, 230026, P.R. China}
	
	\email{wxt2020@mail.ustc.edu.cn}
	
	\address{L. Xu: CAS Wu Wen-Tsun Key Laboratory of Mathematics, University of Science and Technology of China, Hefei, Anhui, 230026, P.R. China}
	
	\email{leoasa@mail.ustc.edu.cn}


	
	\thanks{}
	\begin{abstract} Let $G$ be an infinite discrete countable group and  $(X,G)$ be  a minimal $G$-system. In this paper, we prove  the  supremum of topological sequence entropy  of $(X,G)$ is not less than $\log(\sum_{\mu\in\mathcal{M}^e(X,G)}e^{h_\mu^*(X,G)})$. If additionally $G$ is abelian then there is a constant $K\in\mathbb{N}\cup\{\infty\}$ with $\log K\le h_{top}^*(X,G)$ such that 
		$\nu(\{y\in H:|\pi^{-1}(y)|=K\})=1$ where $(H,G)$ is the maximal equicontinuous factor of $(X,G)$,  $\pi:(X,G)\to (H,G)$ is the factor map and $\nu$ is the Haar measure on $H$.
	\end{abstract}
	
	\maketitle

	\section{Introduction}
	Throughout this paper, 	
	we always assume that  $G$ is an infinite discrete countable group, unless otherwise stated. Meanwhile, we always assume that $X$ is a compact metric
	space with a metric $d$ and $\mathcal{B}_X$ is the Borel $\sigma$-algebra of $(X,d)$. Recall that  by a \textbf{$G$-system}  we mean a pair $(X,G)$, where $G$ acts on $X$ as a group of homeomorphisms. By a \textbf{$G$-measure preserving system} we mean a triple $(X,\mu,G)$, where $(X,G)$ is a $G$-system and $\mu$ is a $G$-invariant Borel probability measure on $X$.

	The notion of sequence entropy for a measure preserving map was introduced by Kushnirenko \cite{Ku},  it was shown
	that an ergodic measure preserving transformation has discrete spectrum if and only if
	its sequence entropy is zero for any sequence. After that
	the topological sequence entropy for a $\mathbb{Z}$-system was   introduced by
	Goodman \cite{G}. A system is null if the sequence entropy of the system is zero for any sequence. It was shown that if  a  minimal $\mathbb{Z}$-system is null then it is uniquely ergodic, has discrete spectrum, mean equicontinuous and is an almost one to one extension of its
	maximal equicontinuous factor \cite{HLSY}. The definition of  topological sequence entropy (resp. measure theoretic sequence entropy) can  similarly be extended to $G$-systems (resp. $G$-measure preserving system). The  supremum of  topological sequence entropy for $G$-system is $\log K$ for some $K\in\mathbb{N}\cup\{\infty\}$ \cite{HLSY,HY}. In the case $G$ is amenable, the  supremum of measure theoretic sequence entropy for $G$-system is $\log K$ for some $K\in\mathbb{N}\cup\{\infty\}$ \cite{HY,LY}. In this paper, we will show this is true for all $G$-measure preserving systems (see Theorem \ref{th-3}).
	
	In \cite{Q}, Qiu showed that if a minimal $\mathbb{Z}$-system has  the  supremum of topological sequence entropy $\log K$ for some $K\in\mathbb{N}$, then there is $k\le K$ such that 
	$\nu(\{y\in H:|\pi^{-1}y|=k\})=1$ where $(H,\mathbb{Z})$ is the maximal equicontinuous factor, $\pi$ is the factor map and $\nu$ is the Haar measure on $H$. 
	The following result  generalizes  Qiu's result to  abelian group actions.
	\begin{thm}\label{thm-5}
		Let $G$ be an infinite discrete countable abelian group and  $(X,G)$ be  a minimal  $G$-system. Then there exists $K\in\mathbb{N}\cup\{\infty\}$ with $\log K\le h_{top}^*(X,G)$ such that 	$$\nu(\{y\in H:|\pi^{-1}(y)|=K\})=1,$$
		where $h_{top}^*(X,G)$ is the supremum of topological sequence entropy, $(H,G)$ is the maximal equicontinuous factor of $(X,G)$, $\pi$ is the factor map and $\nu$ is the Haar measure of $H$.
	\end{thm}
	Theorem \ref{thm-5} shows that if an infinite discrete countable abelian group action $(X,G)$ is minimal and has $h_{top}^*(X,G)=\log K$, then it has no more than $K$ different ergodic measures. In fact, it is true for  all minimal $G$-systems.  This conclusion comes from the following theorem.
	\begin{thm}\label{Thm-B}Let $G$ be an infinite discrete countable group and $(X,G)$ be a minimal $G$-system. Then  $h_{top}^*(X,G)\ge\log \sum_{\mu\in\mathcal{M}^e(X,G)}e^{h_\mu^*(X,G)}$, where $\mathcal{M}^e(X,G)$ is the set of ergodic measures on $(X,G)$.
	\end{thm}
	\begin{rem}Theorem \ref{thm-5} is not true for all group actions. For example, put $\mathbb{T}=\mathbb{R}/\mathbb{Z}=[0,1)$. Define $T,S$ on $\mathbb{T}$ by $Tx=x+\alpha$ and $Sx=x^2$ for $x\in[0,1)$ where $\alpha$ is an irrational number. Let $G=\langle T,S\rangle$. Since $(\mathbb{T},T)$ is minimal, so is $(\mathbb{T},G)$. The maximal equicontinuous factor of $(\mathbb{T},G)$ is trivial and the size of the fiber is $\infty$. However Glasner showed  in \cite{G2}  that  $(\mathbb{T},G)$ is null.
	\end{rem}
	
	The concept of tameness was introduced by K\"ohler in \cite{K}. Here we follow the definition
	of Glasner \cite{G3}. A topological dynamical system is said to be tame if its enveloping
	semigroup is separable and Fr\'echet, and it is said to be non-tame otherwise. It is known
	that a minimal null system is tame \cite{KH}. A structure theorem for an almost automorphic tame system
	has been established by Fuhrmann, Glasner, J\"ager and Oertel in \cite{FG}, i.e., an almost automorphic tame system is regular. Here an almost automorphic topological group action means a minimal group action which is an almost one-to-one extension of its maximal equicontinuous factor. If a minimal tame group action $(X, G)$ admits an invariant probability measure, Glasner showed that ($X,G)$ is almost automorphic \cite{G1}. In this paper, we show a similar result for all tame  minimal $G$-systems.
	\begin{thm}\label{thm-A}Suppose that $(X, G)$ is a tame minimal $G$-system. Let $(H,G)$ be the maximal equicontinuous factor of $(X,G)$ and $\pi$ be the factor map. Assume $K\in\mathbb{N}$ and  $\pi$ is almost $K$ to one. Then 
		$\nu(\{h\in H:|\pi^{-1}(h)|=K\})=1$ where $\nu$ is the Haar measure of $H$.
	\end{thm}
	
	Let $(X, G)$ be a $G$-system and $K\in\mathbb{N}$. We say that $(X, G)$ has no $K$-IT-tuple if for any
	tuple of  disjoint nonempty open subsets $(U_1, U_2,\cdots, U_K)$ there is not an infinite
	independence set for them, i.e. for any infinite set $S\subset G$, there is some $a\in \{1, 2,\cdots, K\}^S$
	such that
	$\cap_{t\in S}t^{-1}U_{a_t}=\emptyset.$ In \cite{KH}, Kerr and Li showed that a topological dynamical system is tame if and only if it has no $2$-IT-tuple. Therefore, Theorem \ref{thm-A} is immediately from the following theorem.
	\begin{thm}\label{thm-h} Suppose that $(X, G)$ is a minimal $G$-system. Let $(H,G)$ be the maximal equicontinuous factor of $(X,G)$ and $\pi$ be the factor map. Let $K_1\in\mathbb{N}\cup\{\infty\},K_2\in\mathbb{N}$ and $(X,G)$ satisfy the following conditions:
		\begin{itemize}
			\item[(1)]$\nu(\{y\in H:|\pi^{-1}(y)|=K_1\})=1$ where $\nu$ is the Haar measure of $H$;
			\item[(2)]$\pi$ is almost $K_2$ to one.
		\end{itemize}
		If $\lceil\frac{K_1}{K_2}\rceil\ge 2$,	then $(X,G)$ has non-trivial $K$-IT-tuple for $2\le K\le \lceil\frac{K_1}{K_2}\rceil, K\in\mathbb{N}$.
	\end{thm}
	In \cite{HLSY2}, Huang, Lian, Shao and Ye showed that if a minimal system  under an amenable group action is an almost $N$ to one extension of its maximal equicontinuous factor and has no $K$-IT-tuple, then it has no more than $N(K-1)^N$ different ergodic measures. By using Theorem \ref{thm-h}, we can improve this result.
	\begin{thm}\label{thm-B}Suppose that $(X, G)$ is a minimal $G$-system. Let $(H,G)$ be the maximal equicontinuous factor of $(X,G)$ and $\pi$ be the factor map. Assume that $(X,G)$ satisfies the following conditions:
		\begin{itemize}
			\item[(1)] $\pi$ is almost $K_2$ to one for some $K_2\in\mathbb{N}$;
			\item[(2)] $(X,G)$ has no non-trivial  $K$-IT-tuple  for some $K\in\mathbb{N},K\ge 2$.
		\end{itemize}
		Then $(X,G)$ has no more than $K_2(K-1)$ different ergodic measures.
	\end{thm}

	The structure of the paper is as follows. In Section 2, we recall some basic notions and results. In Section 3, we prove  Theorem  \ref{thm-5}. In Section 4, we prove Theorem \ref{Thm-B}. In section 5, we prove Theorem \ref{thm-A}, Theorem \ref{thm-h} and Theorem \ref{thm-B}.
	\section{Preliminary}
	In this section, we review some basic notions and fundamental properties of $G$-systems.
	\subsection{Regionally	proximal relation}
	A $G$-system $(X,G)$ is called minimal if $X$ contains no proper non-empty closed invariant
	subsets. It is easy to verify that a $G$-system is minimal if and only if every orbit is dense. A factor map $\pi: X\to Y$ between two $G$-systems $(X,G)$ and $(Y,G)$ is a continuous onto map
	which intertwines the actions. In this case,  we say that $(Y,G)$ is a factor of $(X,G)$ or $(X,G)$ is an
	extension of $(Y,G)$.
	
	A $G$-system $(X,G)$ is equicontinuous if for any $\epsilon > 0$, there is  $\delta> 0$ such that whenever
	$x, y\in X$ with $d(x, y) < \delta$, then $d(gx,gy) < \epsilon$ for all $g\in G$. Let $(X,G)$ be a $G$-system. There is
	a smallest invariant equivalence relation $S_{eq}$ such that the quotient system $(X/S_{eq},G)$ is
	equicontinuous \cite{EG}. The equivalence relation $S_{eq}$ is called the equicontinuous structure
	relation and the factor $(X/S_{eq},G)$ is called the maximal equicontinuous factor of
	$(X,G)$. We remark that $S_{eq}$ is the smallest closed $G$-invariant equivalent relation containing the $2$-regionally proximal relation $RP_2(X,G)$ where 	\begin{align*}
		RP_2(X, G)=\{&(x_1,x_2)
		\in X^2 :\text{ for any }\epsilon > 0\text{ there exist }x_1',x_2'
		\in X\text{ and }g \in G\\
		&\text{ with } d(x_i, x_i')\le \epsilon  (i=1,2),
		\text{ and }d(gx_1', gx_2')\le \epsilon \}.
	\end{align*}
	The $2$-regional proximal relation $RP_2(X,G)$ is $G$-invariat, closed, symmetric and reflective. It turns that in many cases  $RP_2(X,G)$ is an equivalence relation. For example, $G$ is an abelian group and $(X,G)$ is minimal. Hence,  $(X/RP_2(X,G),G)$ is the maximal equicontinuous factor of $(X,G)$  when $G$ is abelian and $(X,G)$ is minimal.
	
	For $k\ge 2$, the $k$-regionally
	proximal relation is  defined by
	\begin{align*}
		RP_k(X, G)=\{&(x_i)_{i=1}^k
		\in X^k :\text{ for any }\epsilon > 0\text{ there exist }x_i'
		\in X\text{ and }g \in G\\
		&\text{ with } d(x_i, x_i')\le \epsilon  (1 \le  i \le k),
		\text{ and }d(gx'_i, gx_j')\le \epsilon (1\le i\le j\le k)\}.
	\end{align*}

	The notion of $k$-regionally proximal relation was firstly introduced  in \cite{HLY}. The following theorem  was proved in
	\cite[Theorem 8]{A}.
	\begin{thm}\label{rp}Let $(X,G)$ be a minimal $G$-system and $n\ge 2$. Assume that $X\times X$ has a dense set of minimal points. If $x_1, x_2,\cdots, x_n \in X$ satisfy $(x_1, x_i)\in RP_2(X,G)$
		for $1\le i\le n$, then $(x_1, x_2,\cdots , x_n)\in RP_n(X,G)$.
	\end{thm}
	Especially, if $G$ is abelian and $(X,G)$ is a minimal $G$-system, then $(x,gx)$ is minimal for all $x\in X$ and $g\in G$. That is,  $X\times X$ has a dense set of minimal points and then Theorem \ref{rp} is true  in this case (In the case $G$ is abelian, one can see \cite{HLY} or \cite{Q} for other proofs).
	
	\subsection{Independence set}
	One may use independence sets to give an equivalent definition of tameness.
	\begin{defn}Let $(X, G)$ be a $G$-system and $k \in\mathbb{N} $. For a tuple $A=(A_1,\cdots, A_k)$ of subsets of $X$,
		we say that a set $J\subset G$ is an independence set for $A$ if for every nonempty finite subset
		$I\subset J$ and function $\sigma : I\to  \{1, 2,\cdots, k\}$ we have
		$$\cap_{s\in I}s^{-1}A_{\sigma(s)}\neq\emptyset.$$
	\end{defn} 
	\begin{defn}Let $(X, G)$ be a $G$-system and $n\ge  2$. We call a tuple $x=(x_1,\cdots, x_n)\in X^n$ is an IT-tuple of length $n$ (or an IT-pair if $n = 2$) if for any product neighborhood $U_1\times U_2\times\cdots \times U_n$
		of $x$ in $X^n$ the tuple $(U_1,U_2,\cdots ,U_n)$ has an infinite independence set. We denote the
		set of IT-tuples of length $n$ by $\text{IT}_n(X,G)$.
	\end{defn}
	
	The diagonal of $X^n$ is defined by
	$$\triangle_{n}(X)= \{(x,\cdots, x)\in X^n : x\in X\},$$
	and put
	$$\triangle^{(n)}(X)= \{(x_1,x_2,\cdots, x_n)\in X^n :\text{ for some }i \neq j, x_i = x_j\}.$$
	When $n = 2$ one writes $\triangle(X) = \triangle_2(X) = \triangle^{(2)}(X)$.
	
	\begin{prop}\label{p-ind}Let (X, G) be a $G$-system  and $n\ge  2$.
		\begin{itemize}
			\item[(1)] Let $(A_1,\cdots, A_n)$ be a tuple of non-empty closed subsets of $X$ which has an infinite
			independence set. Then there exists a $n$-IT-tuple $(x_1,\cdots, x_n)$ with $x_j\in A_j$ for all
			$1\le  j \le  n$.
			\item[(2)] $\text{IT}_2(X,G)\setminus \triangle_2(X)$ is nonempty if and only if $(X, G)$ is non-tame.
			\item[(3)] $\text{IT}_n(X, G)$ is a closed $G$-invariant subset of $X^n$. 
			\item[(4)] Let $\pi : (X, G)\to  (Y, G)$ be a factor map. Then $\pi^{(n)}
			(\text{IT}_n(X, G)) = \text{IT}_n(Y, G)$, where
			$\pi^{(n)}: X^n \to  Y^ n$ defined by $\pi^{(n)}(x_1,\cdots, x_n)=(\pi(x_1),\cdots, \pi(x_n)).$
			\item[(5)] Suppose that $Z$ is a closed $G$-invariant subset of $X$. Then $\text{IT}_n(Z, G)\subset \text{IT}_n(X, G)$.
		\end{itemize}
	\end{prop}
	
	\subsection{Ellis semigroup} An Ellis semigroup is a compact space with a semigroup multiplication which is continuous in only one variable. For an Ellis semigroup $E$, the element $u$ with $u^2=u$ is called an
	idempotent. The Ellis-Numakura theorem says that for any Ellis semigroup $E$, the
	set $J(E)$ of idempotents of $E$ is not empty \cite{El}. A non-empty subset $I$ of $E$ is a left ideal  (resp. right ideal) if $EI\subset I$ (resp. $IE\subset I$). A minimal left ideal is the
	left ideal that does not contain any proper left ideal of $E$. Obviously every left ideal is a
	semigroup and every left ideal contains some minimal left ideal.
	
	An idempotent $u\in J(E)$ is minimal if $v \in J(E)$ and $vu= v$ implies $uv= u$. The
	following results are well known \cite{E0,Fu}. Let $L$ be a left ideal of a Ellis semigroup $E$
	and $u\in J(E)$. Then there is some idempotent $v$ in $Lu$ such that $uv = v$ and $vu = v$; an
	idempotent is minimal if and only if it is contained in some minimal left ideal.

	Given a $G$-system $(X,G)$, the Ellis semigroup $E(X)$ associated to
	$(X, G)$ is defined as the closure of $\{x\to  tx: t \in G \}\subset X^X$ in the product topology, where
	the semi-group operation is given by the composition. On $E(X)$, we may consider the
	$G$-system given by $E(X)\ni s\to  ts$ for each element $t\in G$.
	\begin{thm}[\cite{A1}, pp. 52-53]\label{thm-A1}Suppose $H$ is a compact metric space and $(H, G)$ is minimal
		and equicontinuous. Then $E(H)$ is a compact metrisable topological group. Further, we have
		the following.
		\begin{itemize}
			\item[(a)] If $G$ is abelian, then $E(H)$ is abelian and $(H, G)$ is isomorphic to $(E(H), G)$.
			\item[(b)]For general $G$, $(H, G)$ is a factor of $(E(H), G)$, where the factor map $\pi$ is given by
			$$\pi: E(H)\to  H, t\to  th$$
			for some fixed $h\in H$. Moreover, $\pi$ is open.
		\end{itemize}
	\end{thm}
	The following is in \cite[Proposition 3.2]{HLSY2} which was proved to be true for
	$k = 2$ in \cite[Proposition 3.3]{FG}.
	\begin{prop}\label{p-eliss} Let $H$ be a locally compact second countable Hausdorff topological group
		with left Haar measure $\nu$, and let $k\in\mathbb{N}$ with $k\ge  2$. Suppose that $V_1,\cdots, V_k\subset H$ are
		compact subsets that satisfy
		\begin{itemize}
			\item[(1)]$\text{cl}(\text{int}(V_i))=V_i$ for $i=1,2,\cdots,k$;
			\item[(2)]$\text{int}(V_i)\cap\text{int}(V_j )=\emptyset$ for all $1\le i< j \le k$;
			\item[(3)]$\nu(\cap_{i=1}^k V_i)>0$.
		\end{itemize}
		Further, assume that $G\subset H$ is a dense subgroup and $\mathcal{G}\subset H$ is a residual set. Then there exists an infinite set $I\subset G$ such that for all $a\in\{1, 2, ..., k\}^I$ there exists $h\in \mathcal{G}$ with the property that
		$$h\in \bigcap_{t\in I}t^{-1}\text{int}(V_{a_t}).$$
	\end{prop}
	Now we introduce another type of Ellis semigroups.
	\begin{defn}
		A filter on $G$ is a nonempty family $\mathfrak{p}$ of subsets of $G$ with the following properties
		\begin{enumerate}
			\item If $A,B\in\mathfrak{p}$ then $A\cap B\in\mathfrak{p}$;
			\smallskip
			\item If $A\in\mathfrak{p}$ and $A\subset B\subset G$, then $B\in\mathfrak{p}$;
			\smallskip
			\item $\emptyset\notin \mathfrak{p}$.
		\end{enumerate}
		If in addition, $\mathfrak{p}$ satisfies that for any $A\subset G$, either $A\in\mathfrak{p}$ or $G\setminus A\in\mathfrak{p}$, then  it is called an ultrafilter on $G$ \cite[Section 3]{Hi}.
	\end{defn}
	
	Let
	$$\beta G:=\{\mathfrak{p}:\mathfrak{p}\text{ is an ultrafilter on }G\}.$$
	The sets with the form
	$H:=\{\mathfrak{p}\in\beta G: H\in\mathfrak{p}\} $ with $H\subset G$ form a base for a topology on $\beta G$. With
	this topology $\beta G$ becomes a compact Hausdorff space. The semigroup operation on $ G$ is defined by 
	\[A\in\mathfrak{p}\cdot\mathfrak{q}\Leftrightarrow\{x\in G:Ax^{-1}\in\mathfrak{p}\}\in\mathfrak{q},\]
	where $Ax^{-1}:=\{y\in G:yx\in A\}$. Then $\beta G$ is an Ellis semigroup.
	
	Let $(U_g)_{g\in G}$ be a unitary representation of a group $G$ on a
	Hilbert space $\mathcal{H}$, and
	$$\mathcal{H}_c=\{\varphi\in\mathcal{H}:\{U_g \varphi:g\in G\}\text{ is precompact in }\mathcal{H}\}.$$
	Since the unit ball in $\mathcal{H}$ is a compact metrizable space with
	respect to the weak topology, the expression $\mathfrak{p}$-$\lim_{g\in G} U_g\varphi$ has a well defined meaning.
	The following consequence was proved by Bergelson \cite[Theorem 4.4 and Theorem 4.5]{Be}.
	\begin{thm}\label{thm:weak mixing}
		Let $(U_g)_g\in G$ be a unitary representation of a group $G$ on a
		Hilbert space $\mathcal{H}$. 	Let $\mathfrak{p}\in\beta G$ be a minimal idempotent. Then 
		$$\mathcal{H}=\mathcal{H}_c\oplus\mathcal{H}_{wm},$$
		where $\mathcal{H}_{wm}:=\{\varphi\in\mathcal{H}:\mathfrak{p}\text{-}\lim_{g\in G}U_g\varphi=0\}$.
	\end{thm}
	
	\subsection{Measure decomposition}
	Let $X$ be a compact metric space. Denote by $\mathcal{B}_X$ the Borel $\sigma$-algebra of $X$ and $\mathcal{M}(X)$ be the set of all Borel probability measures on $X$. For $\mu\in \mathcal{M}(X)$,  denote by $\mathcal{B}_X^\mu$
	the completion of $\mathcal{B}_X$ under $\mu$ and denote by $\text{supp} \mu$ the support of $\mu$, i.e. the smallest closed subset of $X$ with full measure. Under the weak$^*$ topology, $\mathcal{M}(X)$ is a nonempty compact convex space. 	For a sub-$\sigma$-algebra  $\mathcal{F}$ of  $\mathcal{B}_X^\mu$,
	$\mu$ can be disintegrated over $\mathcal{F}$ as $\mu = \int_X
	\mu_{\mathcal{F},x}d\mu(x)$ where $\mu_{\mathcal{F},x} \in \mathcal{M}(X)$. The disintegration is characterized by the properties
	\eqref{1} and \eqref{2} below:
	\begin{align}\label{1}\text{for every }f \in L^1(X, \mathcal{B}_X^\mu, \mu), f \in L^1(X, \mathcal{B}_X^\mu, \mu_{\mathcal{F},x}) \text{ for } \mu\text{-a.e. }x \in X,
	\end{align}
	and the map $x\to \int_Xfd\mu_{\mathcal{F},x}$
	is in $L^1(X, \mathcal{F}, \mu)$;
	\begin{align}\label{2}\text{for every }f \in L^1(X, \mathcal{B}_X^\mu, \mu),\mathbb{E}_\mu( f|\mathcal{F})(x)= \int_Xfd\mu_{\mathcal{F},x} \text{ for } \mu\text{-a.e. }x \in X.
	\end{align}
	If $\mathcal{F}',\mathcal{F}$ are sub-$\sigma$-algebras of $\mathcal{B}_X^\mu$ with $\mathcal{F}'\subset \mathcal{F}$ then for every $f \in L^1(X, \mathcal{B}_X^\mu, \mu)$, $$\mathbb{E}_\mu( f|\mathcal{F}')(x)=\mathbb{E}_\mu(\mathbb{E}_\mu( f|\mathcal{F})|\mathcal{F}')(x) \text{ for } \mu\text{-a.e. }x \in X.$$
	That is, 
	\begin{align*}
		\int_Xfd\mu_{\mathcal{F}',x}= \int_X \int_Xf d\mu_{\mathcal{F},y} d\mu_{\mathcal{F}',x}(y)\text{ for all }f\in L^1(X, \mathcal{B}_X^\mu, \mu)\text{ for } \mu\text{-a.e. }x \in X.
	\end{align*}
	Then 
	\begin{align}\label{e-19-01}
		\mu_{\mathcal{F}',x}=  \int_X\mu_{\mathcal{F},y} d\mu_{\mathcal{F}',x}(y)\text{ for } \mu\text{-a.e. }x \in X.
	\end{align}

	Let $(X,G)$ be  a $G$-system. We say $\mu \in \mathcal{M}(X)$ is $G$-invariant if $\mu (g^{-1} B) = \mu (B)$ holds for all $B \in \mathcal{B}_X^\mu$ and $g \in G$. Denote by $\mathcal{M}(X,G)$ the set of  $G$-invariant Borel probability measures of $(X,G)$.
	We say $B \in \mathcal{B}_X^\mu$ is $G$-invariant if $g^{-1} B = B$ for all $g \in G$,
	and $\mu \in \mathcal{M}(X,G)$ is ergodic if for any $G$-invariant Borel set $B \in \mathcal{B}_X^\mu$, $\mu (B) = 0$ or $\mu(B) = 1$ holds. Denote by $\mathcal{M}^e(X,G)$ the set of ergodic measures of $(X,G)$. 
	Since $G$ is countable,	it is well known that   $\mathcal{M}^e(X,G)$ is the collection of all extreme points of $\mathcal{M}(X,G)$.  Using Choquet representation theorem, for each $\mu \in \mathcal{M}(X,G)$ there is a unique measure $\tau$ on the Borel subsets of the  compact space $\mathcal{M}(X,G)$ such that $\tau (\mathcal{M}^e(X,G)) =1$ and $\mu = \int_{\mathcal{M}^e(X,G)} m \mathrm{d} \tau (m)$, which is called the ergodic decomposition of $\mu$.
	If we let $\mathcal{I}_\mu(G)=\{B\in\mathcal{B}_X^\mu:gB=B,\forall g\in G\}$, the ergodic decomposition of $\mu$ is just the  disintegration of $\mu$ over $\mathcal{I}_\mu(G)$.
	
	\subsection{Sequence entropy}Let $X$ be a compact metric space. A cover of $X$ is a finite family of subsets of $X$ whose union is $X$. A partition of $X$ is a cover of $X$ whose elements are pairwise disjoint. Denote by $\mathcal{C}_X$ (resp. $\mathcal{C}^o_X$) the set of all finite covers (resp. finite open covers) of $X$ and by $\mathcal{P}_X$ the set of all finite partitions of $X$.
	
	Let $(X,G)$ be a $G$-system.
	Let $\mathcal{A}=(g_n)_{n=1}^\infty$ be a sequence of $G$ and $\mathcal{U}\in \mathcal{C}^o_X$. The topological sequence entropy of $\mathcal{U}$ with respect to $(X,G)$ along $\mathcal{A}$ is defined by 
	$$h_{top}^\mathcal{A}(G,\mathcal{U})=\limsup_{n\to\infty}\frac{1}{n}\log N(\bigvee_{i=1}^ng_i^{-1}\mathcal{U}),$$ 
	where $N(\bigvee_{i=1}^ng_i^{-1}\mathcal{U})$ is the minimal cardinality among all cardinalities of sub-covers of
	$\bigvee_{i=1}^ng_i^{-1}\mathcal{U}$. The topological sequence entropy of $(X,G)$ along $\mathcal{A}$ is 
	$$h_{top}^\mathcal{A}(X,G)=\sup_{\mathcal{U}\in \mathcal{C}^o_X}h_{top}^\mathcal{A}(G,\mathcal{U}).$$
	Define the  supremum of topological sequence entropy  of $(X,G)$ by
	$$h_{top}^\mathcal{*}(X,G)=\sup_{\mathcal{A}}h_{top}^\mathcal{A}(X,G),$$
	where the supremum is taken over all sequences of $G$. 
	
	By the ideas of Bowen \cite{B} and Dinaburg \cite{D}, the topological sequence entropy along $\mathcal{A}=(g_n)_{n=1}^\infty$  can also be
	defined through the separated set. A subset $E$ of  $X$ is a $(n, \epsilon)$-separated set along $\mathcal{A}$ if
	for any distinct $x, y\in E$ there is $j\in\mathbb{N}$ with $1\le  j\le  n$ such that $d(g_jx,g_jy) >\epsilon$.
	Let  $Sep_n^\mathcal{A}(\epsilon)$ be the largest cardinality of any $(n, \epsilon)$-separated set along $\mathcal{A}$. Then 
	$$h_{top}^\mathcal{A}(X,G)=\lim_{\epsilon\to 0}\limsup_{n\to\infty}\frac{1}{n}\log Sep_n^\mathcal{A}(\epsilon).$$
	
	For $\mu\in\mathcal{M}(X)$, define
	$$\mathcal{P}_X^\mu=\{
	\alpha\in \mathcal{P}_X : \text{each element in }\alpha \text{ belongs to }\mathcal{B}_X^\mu
	\}.$$
	Given $\alpha\in\mathcal{P}_X^\mu$
	and a sub-$\sigma$-algebra $\mathcal{A}$ of $\mathcal{B}_X^\mu$, define
	$$H_\mu(\alpha|\mathcal{A}) = \sum_{A\in\alpha}\int_X-\mathbb{E}_\mu(1_A|\mathcal{A})\log \mathbb{E}_\mu(1_A|\mathcal{A})d\mu,$$
	where $\mathbb{E}_\mu(1_A|\mathcal{A})$ is the expectation of $1_A$ with respect to $\mathcal{A}$. One standard fact is that
	$H_\mu(\alpha|\mathcal{A})$ increases with respect to $\alpha$ and decreases with respect to $\mathcal{A}$. Set $\mathcal{N}=\{\emptyset,X \}$ and define
	$$H_\mu(\alpha)=H_\mu(\alpha|\mathcal{N})=-\sum_{A\in \alpha}\mu(A)\log \mu(A).$$ Let $(X,\mu,G)$ be a $G$-measure preserving system. The sequence entropy of $\alpha$ with respect to $(X, \mu, G)$ along $\mathcal{A}$ is 
	$$h_{\mu}^\mathcal{A}(G,\alpha)=\limsup_{n\to\infty}\frac{1}{n}H_\mu( \vee_{i=1}^ng_i^{-1}\alpha).$$  The sequence entropy of $\alpha$ with respect to $(X, \mu, G)$ is  
	$$h_{\mu}^\mathcal{*}(G,\alpha)=\sup_{\mathcal{A}}h_{\mu}^\mathcal{A}(G,\alpha),$$
	where the supremun is taken over all sequences of $G$.
	Define the  supremum of sequence entropy of $(X,\mu,G)$ by 
	$$h_{\mu}^*(G)=\sup_{\alpha\in\mathcal{P}_X^\mu}h_{\mu}^\mathcal{A}(G).$$
	
	For $\mu\in\mathcal{M}(X,G)$, define the Kronecker $\sigma$-algebra of $(X,\mu,G)$  by $$\mathcal{K}_\mu(G)=\{B\in\mathcal{B}_X^\mu:\{U_g1_B:g\in G\}\text{ is precompact in }L^2(\mu)\},$$ 
	where $U_gf(x)=f(gx)$ for $g\in G,f\in L^2(\mu),x\in X$.	It is easy to see that $\mathcal{K}_\mu(G)$ is a $G$-invariant sub-$\sigma$-algebra of $\mathcal{B}_X^\mu$. 
	For $\mu\in\mathcal{M}^e(X,G)$, let $\mu=\int_X\mu_xd\mu(x)$ be the disintegration of $\mu$ with respect to $\mathcal{K}_\mu(G)$. If $\mu$ is ergodic, there is $k_\mu\in\mathbb{N}\cup\{\infty\}$ such that $|\text{supp}\mu_x|\in \mathbb{N}\cup\{\infty\}$. The following theorem shows the relation between sequence entropy and  Kronecker $\sigma$-algebra. 
	\begin{thm}\label{th-3}
		Let $(X,\mu,G)$ be an ergodic $G$-measure preserving system. Let $\mu=\int_X\mu_xd\mu(x)$ be the disintegration of $\mu$ with respect to $\mathcal{K}_\mu(G)$. Then 
		$h^{*}_{\mu}(G)=\log k_\mu$.
	\end{thm}
	
	We will prove the above theorem in Appendix B. In the case $G$ is amenable, one may see \cite{HMY} and \cite{LY} for proofs.

	\section{Proof of Theorem \ref{thm-5}}
	In this section, we give a  useful  method (Proposition \ref{p-1}) to estimate the lower bound of  the  supremum of topological sequence entropy and prove Theorem \ref{thm-5}.
	
	\begin{prop}\label{p-1}Let $(X,G)$ be a $G$-system, $\gamma\in(0,1)$ and $k\in\mathbb{N}\setminus\{1\}$. Assume that there is $\epsilon>0$  such that for $M\in\mathbb{N}$ and  nonempty open subsets $U_1,U_2,\cdots, U_M$ of $X$ there exist  $g\in G$, $\mathcal{M}\subset\{1,2,\cdots,M\}$ and  $x_{m,k}\in U_m$ for  $k\in\{1,2,\cdots,K\}$, $m\in\mathcal{M}$  such that 
		\begin{itemize}
			\item[(1)] $|\mathcal{M}|\ge (1-\gamma)M$;
			\item[(2)] 
			$d(gx_{m,k},gx_{m,k'})>\epsilon,\text{ for } 1\le k<k'\le K, m\in\mathcal{M}$.
		\end{itemize} 
		Then $h^*_{top}(X,G)\ge \log((1-\gamma)K)$.
	\end{prop}
	\begin{proof}For any finite subset $E$ of $G$, denote by  $\text{Sep}_E(\epsilon)$ the maximal number $L$ such that there is $x_1,x_2,\cdots,x_L\in X$ with $\max_{g\in E}d(gx_l,gx_{l'})>\epsilon$  for all $1\le l<l'\le L$. We are going to show that there is a sequence $\mathcal{A}=\{g_1,g_2,\cdots\}$ of $G$ such that $\text{Sep}_{\{g_1,g_2,\cdots,g_n\}}(\epsilon)\ge (K(1-\gamma))^{n-1}$. 
		
		Let $g_1\in G$. It is clear that  $\text{Sep}_{\{g_1\}}(\epsilon)\ge 1= (K(1-\gamma))^{1-1}$. Now assume that we have $g_1,g_2,\cdots,g_n\in G$ for some $n\in\mathbb{N}$ with $\text{Sep}_{\{g_1,g_2,\cdots,g_n\}}(\epsilon)\ge(K(1-\gamma))^{n-1}$. Then there is nonempty open subsets $U_1,U_2,\cdots,U_{\text{Sep}_{\{g_1,g_2,\cdots,g_n\}}(\epsilon)}$ of $X$ such that 
		$$\max_{1\le i\le n}d(g_iU_m,g_iU_{m'})>\epsilon, \text{ for }1\le m<m'\le \text{Sep}_{\{g_1,g_2,\cdots,g_n\}}(\epsilon).$$
		By assumption, there is  $g_{n+1}\in G$, $\mathcal{M}\subset\{1,2,\cdots,\text{Sep}_{\{g_1,g_2,\cdots,g_n\}}(\epsilon)\}$ and  $x_{m,k}\in U_m$ for  $k\in\{1,2,\cdots,K\}$, $m\in\mathcal{M}$  such that 
		\begin{itemize}
			\item[(1)] $|\mathcal{M}|\ge (1-\gamma)\text{Sep}_{\{g_1,g_2,\cdots,g_n\}}(\epsilon)$;
			\item[(2)] 
			$d(g_{n+1}x_{m,k},g_{n+1}x_{m,k'})>\epsilon,\text{ for } 1\le k<k'\le K, m\in\mathcal{M}$.
		\end{itemize} 
		Then for $  k,k'\in \{1,2,\cdots,K\}, m,m'\in\mathcal{M}$ with $(m,k)\neq (m',k')$, one has 
		$$\max_{1\le i\le n+1}d(g_ix_{m,k},g_ix_{m',k'})>\epsilon.$$
		This implies $\text{Sep}_{\{g_1,g_2,\cdots,g_n,g_{n+1}\}}(\epsilon)\ge (K(1-\gamma))^{n}$. By induction, we get the required sequence $\mathcal{A}=\{g_1,g_2,\cdots\}$ of $G$. Therefore	$$h^*_{top}(X,G)\ge h^\mathcal{A}_{top}(X,G)\ge\limsup_{n\to\infty}\frac{\log \text{Sep}_{\{g_1,g_2,\cdots,g_n,g_{n+1}\}}(\epsilon)}{n}\ge \log( K(1-\gamma)).$$
		This  ends the proof of Proposition \ref{p-1}.
	\end{proof}	
	Now we prove Theorem \ref{thm-5}.
	\begin{proof}[Proof of Theorem \ref{thm-5}]
		Let $G$ be an infinite discrete countable abelian group and  $(X,G)$ be  a minimal  $G$-system.  Let $(H,G)$ be the maximal equicontinuous factor of $(X,G)$, $\pi$ be the factor map and $\nu$ be the uniquely ergodic measure of $(H,G)$.

		Since $(H,\nu,G)$ is ergodic, there is $P \in\mathbb{N}\cup\{\infty\}$ such that $|\pi^{-1}(y)|=P$ for $\nu$-a.e. $y\in H$. Hence the set $\{y\in H:|\pi^{-1}(y)|\ge K\}$ has measure $0$ or $1$  for $K\in\mathbb{N}$. Fix $K\in\mathbb{N}$ with	$\nu(\{y\in H:|\pi^{-1}(y)|\ge K\})=1$. To prove Theorem \ref{thm-5}, it is enough to show that 
		$ h_{top}^*(X,G)\ge \log K$.

		Fix $\gamma\in(0,1)$. There is $\epsilon>0$ such that the set
		\begin{align}\label{e-11}W:=\{y\in H:\text{Sep}(\pi^{-1}(y),2\epsilon)\ge K\}
		\end{align}
		has measure larger than $1-\gamma$ where $\text{Sep}(\pi^{-1}(y),2\epsilon)$ is the maximal number $L$ such that there is $x_1,x_2,\cdots,x_L\in \pi^{-1}(y)$ with $d(x_l,x_{l'})>2\epsilon$  for all $1\le l<l'\le L$.
		
		Now we verify the conditions in Proposition \ref{p-1}. Fix $M\in \mathbb{N}$ and nonempty open subsets $U_1,U_2,\cdots,U_M$ of $X$. Since $(X,G)$ is minimal, there is a nonempty open subset $U$ of $X$ and $g_1,g_2,\cdots,g_M\in G$ such that 
		$$g_mU\subset U_m, \text{ for all } m\in\{1,2,\cdots,M\}.$$
		Notice that 
		$$\int_{H} \sum_{m=1}^M1_{W}(g_my)d\nu(y)=M\nu(W)\ge (1-\gamma)M.$$ There is $y_0\in H$ such that 
		$$|\{m\in\{1,2,\cdots,M\}:g_my_0\in W\}|\ge (1-\gamma)M.$$
		Put
		$$\mathcal{M}=\{m\in\{1,2,\cdots,M\}:g_my_0\in W\}.$$
		By \eqref{e-11}, there exists $z_{mk}\in \pi^{-1}(g_my_0),m\in\mathcal{M},k=1,2,\cdots,K$ such that 
		$$d(z_{mk},z_{mk'})>2\epsilon,\text{ for }m\in\mathcal{M}, 1\le k<k'\le K.$$ 
		Fix $\delta>0$ such that 
		\begin{align}\label{e0}x,x'\in X,d(x,x')<\delta\text{ implies }d(g_mx,g_mx')<\frac{\epsilon}{2},\text{ for all }m\in\mathcal{M}.
		\end{align}
		Notice that 
		$$\{g_m^{-1}z_{mk}:m\in\mathcal{M},k\in\{1,2,\cdots,K\}\}\subset \pi^{-1}(y_0).$$
		Since $G$ is abelian and $(X,G)$ is minimal, one has  $\pi^{-1}(y_0)\times\pi^{-1}(y_0)\subset RP_2(X,G)$.	By Theorem \ref{rp}, 
		$(g_m^{-1}z_{mk})_{m\in\mathcal{M},k\in\{1,2,\cdots,K\}}\in RP_{|\mathcal{M}|K}(X,G).$	
		There is $h\in G$ and $z_{mk}'\in B(g_m^{-1}z_{mk},\delta)$ for $ m\in\mathcal{M},k\in\{1,2,\cdots,K\}$ such that 
		$$hz_{mk}'\in U, \text{ for } m\in\mathcal{M},1\le k\le K.$$ Put $x_{mk}=g_mh z_{mk}', \text{ for } m\in\mathcal{M},1\le k\le K.$ Then
		$$x_{mk}=g_mh z_{mk}'\in g_mU\subset U_m$$
		for $m\in\mathcal{M},k\in\{1,2,\cdots,K\}$
		and since each $z_{mk}'\in B(g_m^{-1}z_{mk},\delta)$ one has by \eqref{e-11} and \eqref{e0} that
		\begin{align*}d(h^{-1}x_{mk},h^{-1}x_{mk'})&=d(g_mz_{mk}',g_mz_{mk'}')\\
			&\ge d(z_{mk},z_{mk'})-d(g_mz_{mk}',z_{mk})-d(g_mz_{mk'}',z_{mk'})\\
			&>\epsilon\end{align*}
		for $m\in\mathcal{M},1\le k<k'\le K$. Therefore, the condition in Proposition \ref{p-1} is valid for $(X,G)$.
		
		By Proposition \ref{p-1}, $h_{top}^*(X,G)\ge \log(1-\gamma)K$.
		Let $\gamma\to 0$. One has 
		$$h_{top}^*(X,G)\ge \log K.$$
		This ends the proof of Theorem \ref{thm-5}.
	\end{proof}
	
	\section{Proof of Theorem \ref{Thm-B}}
	In this section, we are going to prove Theorem \ref{Thm-B}. To do this, we need some lemmas firstly.
	\begin{lem}\label{lem-1}Let $k\in\mathbb{N}\setminus\{1\}$ and  $\epsilon>0$. Let $L_i,i=1,2,\cdots,k$ be subsets of $X$ with
		$$|L_i|=k, d(x,y)>2\epsilon,\text{ for } x,y\in L_i,x\neq y,$$
		for each $i$.	Then there exists $x_i\in L_i$ for $i=1,2,\cdots,k$ such that
		$$d(x_{i},x_{i'})>\epsilon,\text{ for } 1\le i<i'\le k.$$ 
	\end{lem}
	\begin{proof}If $k=2$, we let $x_1\in L_1$. Since $|L_2|=2, d(x,y)>2\epsilon,\text{ for all } x,y\in L_2,x\neq y,$ there is $x_2\in L_2$ such that $d(x_1,x_2)>\epsilon$. Hence the lemma is true for $k=2$.
		
		Now we assume that the lemma is true for some $k\ge 2$. Then for $k+1$, there must exist $L_i'\subset L_i,i\in\{1,2,\cdots,k\}$ such that 
		$$|L_i'|=k, d(x,y)>2\epsilon, \text{ for } x,y\in L_i',x\neq y.$$
		By induction, there is $x_i\in L_i',i\in\{1,2,\cdots,k\}$ such that 
		$$d(x_{i},x_{i'})>\epsilon,\text{ for } 1\le i<i'\le k.$$ 
		Note that $|L_{k+1}|=k+1$. Then there exists  $x_{k+1}\in L_{k+1}$, such that  
		$$d(x_{i},x_{i'})>\epsilon,\text{ for } 1\le i<i'\le k+1.$$ 
		We finish the proof of Lemma \ref{lem-1}.
	\end{proof}
	For $\delta>0$ and $K\in\mathbb{N}\setminus\{1\}$, put
	$$X^K_2(\delta)=\{(x_1,x_2,\cdots,x_K)\in X^K:d(x_k,x_{k'})\le\delta\text{ for some }1\le k<k'\le K\}.$$
	\begin{lem}\label{lem}Let $\mu\in\mathcal{M}(X),\delta>0$ and $K\in\mathbb{N}\setminus\{1\}$. 
		If $|\text{supp}\mu|\ge K$, then $\lim_{\delta\to 0}\mu^K(X^K\setminus X^K_2(\delta))>0$.
	\end{lem}
	\begin{proof}Since $|\text{supp}\mu|\ge K$, we can find pairwise distinct $x_1,x_2,\cdots,x_K\in \text{supp}\mu$. 
		Put
		$$r:=\min_{1\le k<k'\le K}d(x_k,x_{k'})>0.$$
		Then each $B(x_k,r/3)$ has positive measure and
		$B(x_1,\frac{r}{3})\times B(x_2,\frac{r}{3})\times \cdots\times B(x_K,\frac{r}{3})\subset X^K\setminus X^K_2(\delta)$ for all $\delta<\frac{r}{3}$. Therefore
		\begin{align*}&
			\lim_{\delta\to 0}\mu^K(X^K\setminus X^K_2(\delta))\\
			&\ge \mu^K(B(x_1,r/3)\times B(x_2,r/3)\times\cdots\times B(x_K,r/3))\\
			&=\Pi_{k=1}^K\mu(B(x_k,r/3))\\
			&>0.
		\end{align*}
		This ends the proof of Lemma \ref{lem}.
	\end{proof}
	\begin{lem}\label{l-7}Let $\mu\in\mathcal{M}(X)$. Let $M\in\mathbb{N}$, $\gamma\in(0,1)$ and  $A_1,A_2,\cdots,A_M$ be  measurable subsets  of $X$ with $\mu(A_i)\ge1-\gamma$ for $i=1,2,\cdots,M$. Then there exists $\mathcal{M}\subset\{1,2,\cdots,M\}$ such that 
		\begin{itemize}
			\item[(1)] $|\mathcal{M}|\ge (1-\gamma)M$;
			\item[(2)] $\mu(\cap_{m\in\mathcal{M}}A_m)>0$.
		\end{itemize}
	\end{lem}
	\begin{proof}Since \begin{align*}1-\gamma\le \frac{1}{M}\sum_{m=1}^M\mu(A_m)=\int_X\frac{1}{M}\sum_{i=1}^M1_{A_m}(x)d\mu(x).
		\end{align*}
it follows that the set $W=\{x\in X:\frac{1}{M}\sum_{i=1}^M1_{A_m}(x)\ge 1-\gamma\}$ has positive measure. Notice  that
		$$W=\bigcup_{\mathcal{M}\subset\{1,2,\cdots,M\}\atop |\mathcal{M}|\ge (1-\gamma)M}\bigcap_{m\in\mathcal{M}}A_m.$$
		Since the set $\{\mathcal{M}\subset\{1,2,\cdots,M\}:|\mathcal{M}|\ge (1-\gamma)M\}$ is finite, we can find a $\mathcal{M}\subset\{1,2,\cdots,M\}$ such that  $\mu(\cap_{m\in\mathcal{M}}A_m)>0$ and $|\mathcal{M}|\ge (1-\gamma)M$. This ends the proof of Lemma \ref{l-7}.
	\end{proof}
	For $\mu\in\mathcal{M}(X,G)$, put
	$$\mathcal{I}_\mu(G)=\{B\in\mathcal{B}_X^\mu:gB=B,\text{ for all } g\in G\}.$$
	The  Kronecker $\sigma$-algebra of $(X,\mu,G)$ is defined by
	$$\mathcal{K}_\mu(G)=\{B\in\mathcal{B}_X^\mu:\{g1_B:g\in G\}\text{ is precompact in }L^2(\mu)\}.$$
	The key to prove Theorem \ref{Thm-B} is the follwing lemma.
	\begin{lem}\label{l-2}For $I\in\mathbb{N}$ and $\mu_1,\mu_2,\cdots,\mu_I\in\mathcal{M}(X,G)$, one has
		$\mathcal{I}_{\mu_1\times\mu_2\times\cdots\times\mu_I}(G)\subset \mathcal{K}_{\mu_1}(G)\times\mathcal{K}_{\mu_2}(G)\times\cdots\times\mathcal{K}_{\mu_I}(G).$
	\end{lem}
	We will prove Lemma \ref{l-2} in Appendix A. Now we are going to prove Theorem \ref{Thm-B}. In fact,  Theorem \ref{Thm-B}
	immediately comes from Theorem \ref{thm-3} and Theorem \ref{th-3}.
	\begin{thm}\label{thm-3}Let $(X,G)$ be a minimal $G$-system. Then $$h_{top}^*(X,G)\ge\log \sum_{\mu\in\mathcal{M}^e(X,G)}k_\mu,$$\
		where $k_\mu$ is as in Theorem \ref{th-3}.\end{thm}
	\begin{proof} Fix $I\in\mathbb{N}$ with $I\le |\mathcal{M}^e(X,G)|$. Assume that $\mu_1,\cdots,\mu_I\in\mathcal{M}^e(X,G)$ are different ergodic measures of $(X,G)$. For each $i$, let $\mu_i=\int_X\mu_{i,x}\mu_i(x)$ be the disintegration of $\mu_i$ with respect to $\mathcal{K}_{\mu_i}(G) $. Let $k_i\in\mathbb{N}$ with $k_i\le k_{\mu_i}$. To prove Theorem \ref{thm-3}, it is enough to show that $h_{top}^*(X,G)\ge\log \sum_{i=1}^Ik_i$.
		
		Put 
		$\overline{X}=X^{\sum_{i=1}^Ik_i}$ and 
		$\overline{\mu}=\mu_1^{k_1}\times\mu_2^{k_2}\times\cdots\times\mu_I^{k_I}.$
		For $\overline{x}\in \overline{X}$, we note 
		$$\overline{x}=(x_1,x_2,\cdots,x_{\sum_{i=1}^Ik_i})\text{ or }\overline{x}=(x_{1,1},\cdots,x_{1,k_1},\cdots,x_{I,k_I})$$
		since there is no confusion. 	For $\overline{x}\in \overline{X}$, we put
		$$\overline{\mu}_{\overline{x}}=\mu_{1,x_{1,1}}\times\cdots\times\mu_{1,x_{1,k_1}}\times\cdots\times\mu_{I,x_{I,{k_I}}}.$$
		For $\epsilon>0$, put
		\begin{align}\label{e-32}W_{\epsilon}=\{\overline{x}\in \overline{X}:\overline{\mu}_{\overline{x}}(X^{\sum_{i=1}^Ik_i}\setminus X^{\sum_{i=1}^Ik_i}_2(\epsilon))>0\}.
		\end{align}
		For $j\in\mathbb{N}$, we let $i(j)$ be the number $i$ such that  $$m\sum_{s=1}^Ik_s+\sum_{s=1}^{i-1}k_s+1\le j\le m\sum_{s=1}^Ik_s+ \sum_{s=1}^ik_s\text{ for some }m\in\mathbb{N}\cup\{0\}.$$
		
		Let $\mathcal{E}$ be the set of $\tau\in \mathcal{M}(\overline{X},G)$ with marginal $\mu_{i(j)}$ on the $j$-th coordinate for $j\in\mathbb{N}$. 
		For $\tau\in \mathcal{E}$ we put 
		$$\tau_\mathcal{K}=\int_{\overline{X}}\overline{\mu}_{\overline{x}}d\tau(\overline{x}).$$ Here $\tau_\mathcal{K}$ is well defined for $\tau\in\mathcal{E}$ since the map $x\to \mu_{i,x}$ is measurable for each $i$. We have the following claim.
		\begin{cl}\label{cl-6}$\lim_{\epsilon\to0}\inf_{\tau\in\mathcal{E}}\tau_{\mathcal{K}}(W_\epsilon)=1.$
		\end{cl}
		The proof of Claim \ref{cl-6} will be presented later. Now
		fix $0<\gamma<1$ and $K=\sum_{i=1}^Ik_i$. By Claim \ref{cl-6}, we can fix $\epsilon>0$ with $\inf_{\tau\in\mathcal{E}}\tau_{\mathcal{K}}(W_\epsilon)\ge 1-\gamma.$ We are going to verify the conditions in Proposition \ref{p-1}. 
		Fix $M\in\mathbb{N}$ and  nonempty open subsets $U_1,U_2,\cdots, U_M$ of $X$. 
		Let 
		$$\overline{\mu}^M=\int_{\overline{X}^M}\overline{\mu}_{\overline{x}_1,\overline{x}_2,\cdots,\overline{x}_M}d\overline{\mu}^M(\overline{x}_1,\overline{x}_2,\cdots,\overline{x}_M)$$ be the disintegration of $\overline{\mu}^M$ with respect to $\mathcal{I}_{\overline{\mu}^M}(G)$ as well as the ergodic decomposition of $\overline \mu^M$. By Lemma \ref{l-2}, $$\mathcal{I}_{\overline{\mu}^M}(G)\subset (\mathcal{K}_{\mu_1}(G)^{k_1}\times\mathcal{K}_{\mu_2}(G)^{k_2}\times\cdots\times\mathcal{K}_{\mu_I}(G)^{k_I})^M.$$
		Then by \eqref{e-19-01}, for $\overline{\mu}^M$-a.e. $(\overline{x}_1,\overline{x}_2,\cdots,\overline{x}_M)\in \overline{X}^M$, one has 
		$$\overline{\mu}_{\overline{x}_1,\overline{x}_2,\cdots,\overline{x}_M}=\int_{\overline{X}^M}\overline{\mu}_{\overline{y}_1}\times\cdots\times \overline{\mu}_{\overline{y}_M}d\overline{\mu}_{\overline{x}_1,\overline{x}_2,\cdots,\overline{x}_M}(\overline{y}_1,\overline{y}_2,\cdots,\overline{y}_M).$$
		Since $(X,G)$ is minimal, $\text{supp}\overline{\mu}^M=X^{M(\sum_{i=1}^Ik_i)}$. One can choose $$\nu:=\overline{\mu}_{\overline{x}_1,\overline{x}_2,\cdots,\overline{x}_M}\in \mathcal{M}^e(\overline{X}^M,G)\text{ for some }(\overline{x}_1,\overline{x}_2,\cdots,\overline{x}_M)\in \overline{X}^M$$ such that  $$\nu( U_1^{\sum_{i=1}^Ik_i}\times U_2^{\sum_{i=1}^Ik_i}\times\cdots\times U_M^{\sum_{i=1}^Ik_i})>0,$$
		\begin{align}\label{e-29} \nu=\int_{\overline{X}^M}\overline{\mu}_{\overline{y}_1}\times\cdots\times \overline{\mu}_{\overline{y}_M}d\nu(\overline{y}_1,\overline{y}_2,\cdots,\overline{y}_M),\end{align}	
		and the projection of $\nu$ on the $j$-th coordinate is $\mu_{i(j)}$.

		For $m\in\{1,2,\cdots,M\}$, put
		$$A_m=\overline{X}^{m-1}\times W_\epsilon\times \overline{X}^{M-m}.$$
		Let $p_m$ be the projection of $\overline{X}^M$ on the $m$-th $\overline{X}$. Then by \eqref{e-29},
		$$p_m(\nu)=\int_{\overline{X}}\overline{\mu}_{\overline{y}}d\nu(\overline{y}).$$
		By Claim \ref{cl-6},
		$$\nu(A_m)=p_m(\nu)(W_\epsilon)>1-\gamma,\text{ for all } m\in\{1,2,\cdots,M\}.$$
		By Lemma \ref{l-7}, there is $\mathcal{M}\subset \{1,2,\cdots,M\}$ such that 
		\begin{itemize}
			\item[(1)] $|\mathcal{M}|\ge (1-\gamma)M$;
			\item[(2)] $\nu(\cap_{m\in\mathcal{M}}A_m)>0$.
		\end{itemize}
		Put $P_m=X^{\sum_{i=1}^Ik_i}$ if $m\notin \mathcal{M}$ and $P_m=X^{\sum_{i=1}^Ik_i}\setminus X^{\sum_{i=1}^Ik_i}_2(\epsilon)$ if $m\in \mathcal{M}$. Then for $(\overline{y}_1,\overline{y}_2,\cdots,\overline{y}_M)\in \cap_{m\in\mathcal{M}}A_m$ one has by \eqref{e-32} that
		$$\overline{\mu}_{\overline{y}_1}\times\cdots\times \overline{\mu}_{\overline{y}_M}(P_1\times P_2\times\cdots\times P_M)=\Pi_{m=1}^M\overline{\mu}_{\overline{y}_m}(P_m)>0.$$
		Since  $\nu(\cap_{m\in\mathcal{M}}A_m)>0$, one has
		\begin{align*}
			&\nu(P_1\times P_2\times\cdots\times P_M)\\&\ge \int_{\cap_{m\in\mathcal{M}}A_m}\overline{\mu}_{\overline{y}_1}\times\cdots\times \overline{\mu}_{\overline{y}_M}(P_1\times P_2\times\cdots\times P_M)d\nu(\overline{y}_1,\cdots,\overline{y}_M)\\
			&>0.
		\end{align*}
		Since $\nu$ is ergodic and  $\nu( U_1^{\sum_{i=1}^Ik_i}\times U_2^{\sum_{i=1}^Ik_i}\times\cdots\times U_M^{\sum_{i=1}^Ik_i})>0,$
		there is $h\in G$ such that 
		$$(h( U_1^{\sum_{i=1}^Ik_i}\times U_2^{\sum_{i=1}^Ik_i}\times\cdots\times U_M^{\sum_{i=1}^Ik_i}))\cap(P_1\times P_2\times\cdots\times P_M)\neq\emptyset.$$
		We can find  $(\overline{x}_1,\overline{x}_2,\cdots,\overline{x}_M)\in U_1^{\sum_{i=1}^Ik_i}\times U_2^{\sum_{i=1}^Ik_i}\times\cdots\times U_M^{\sum_{i=1}^Ik_i}$ such that
		$$h(\overline{x}_1,\overline{x}_2,\cdots,\overline{x}_M)\in P_1\times P_2\times\cdots\times P_M.$$
		If  we write $\overline{x}_m=(x_{mk})_{k=1,2,\cdots,\sum_{i=1}^Ik_i}$ for $m\in \mathcal{M}$, we have
		$$d(hx_{mk},hx_{mk'})>\epsilon,\text{ for } m\in\mathcal{M}, 1\le k<k'\le \sum_{i=1}^Ik_i.$$
		Since each $x_{mk}\in U_m ,\text{ for } m\in\mathcal{M}, 1\le k\le \sum_{i=1}^Ik_i$, the conditions in Proposition \ref{p-1} are valid. Then
		by Proposition \ref{p-1}, $h^*_{top}(X,G)\ge \log ((1-\gamma)\sum_{i=1}^Ik_i)$. Let $\gamma\to 0$. One has $h^*_{top}(X,G)\ge \log\sum_{i=1}^Ik_i$. This ends the proof of Theorem \ref{thm-3} by the beginning argument.
	\end{proof}
	Now we prove Claim \ref{cl-6}.
	\begin{proof}[Proof of Claim \ref{cl-6}] Since $\mu_1,\mu_2,\cdots,\mu_I$ are different ergodic measures of $(X,G)$,
		we can find $\epsilon_0>0$ and  $r_\epsilon\to 0$ as $\epsilon\to0$ and nonempty compact subsets $V_{i,\epsilon}$ of $X$ for $i\in\{1,2,\cdots,I\},\epsilon\in(0,\epsilon_0)$ such that
		\begin{align}\label{e-30} d(V_{i,\epsilon},V_{i',\epsilon})>\epsilon \text{ for } 1\le i< i'\le I\text{ and }\epsilon\in(0,\epsilon_0),\end{align}
		and
		\begin{align*}\mu_{i}(V_{i,\epsilon})>1-r_{\epsilon}\text{ for }1\le i\le I\text{ and }\epsilon\in(0,\epsilon_0). \end{align*}
		Then 
		\begin{align}\label{e-31}\lim_{\epsilon\to 0}\mu_{i}(V_{i,\epsilon})=1\text{ for }1\le i\le I. \end{align}
		
		For $i\in\{1,2,\cdots,I\}$ with $k_i=1$ and $\epsilon\in(0,\epsilon_0),\delta>0$, put
		\begin{align}\label{e-8}V_{i,\epsilon,\delta}= V_{i,\epsilon}\cap\{x\in X:\mu_{i,x}(V_{i,\epsilon})>0\}.\end{align}
		Since $\mu_{i,x}( V_{i,\epsilon})>0$ for $\mu_i$-a.e. $x\in  V_{i,\epsilon}$, one has
		\begin{align}\label{e-9}\lim_{\delta\to 0}\liminf_{\epsilon\to 0}\mu_i(V_{i,\epsilon,\delta})=\lim_{\epsilon\to 0}\mu_i(V_{i,\epsilon})=1.\end{align}
		
		For $i\in\{1,2,\cdots,I\}$ with $k_{i}\ge2$ and $\epsilon\in(0,\epsilon_0),0<\delta<1$, put
		\begin{align}\label{e-1}
			W_{i,\epsilon,\delta}=\{x\in X:\mu_{i,x}^{k_i}(X^{k_i}\setminus X^{k_i}_2(2\epsilon))>1-(1-\delta)^{k_i}\}.
		\end{align}
		Since $|\text{supp}\mu_{i,x}|\ge k_i>1$ for $\mu_i$-a.e. $x\in X$, one has by Lemma \ref{lem} that
		\begin{align}\label{eq-11111}\lim_{\delta\to 0}\lim_{\epsilon\to 0}\mu_i(	W_{i,\epsilon,\delta})=1.
		\end{align}
		Put
		\begin{align}\label{e-111}
			V'_{i,\epsilon,\delta}=\{x\in X:\mu_{i,x}(V_{i,\epsilon})>\delta\}.
		\end{align}
		and 
		\begin{align}\label{e-3}
			V_{i,\epsilon,\delta}=	V'_{i,\epsilon,\delta}\cap W_{i,\epsilon,\delta} .
		\end{align}
		
		Notice that
		\begin{align*}
			\mu_i(V_{i,\epsilon})&=\int_X \mu_{i,x}(V_{i,\epsilon})d\mu_i(x)\\
			&=\int_{V'_{i,\epsilon,\delta}} \mu_{i,x}(V_{i,\epsilon})d\mu_i(x)+\int_{X\setminus V'_{i,\epsilon,\delta}} \mu_{i,x}(V_{i,\epsilon})d\mu_i(x)\\
			&\le \mu_i(V'_{i,\epsilon,\delta})+\delta\mu_i(X\setminus V'_{i,\epsilon,\delta})\\
			&=1-(1-\delta)\mu_i(X\setminus V'_{i,\epsilon,\delta}).
		\end{align*}
		This implies 
		$$\mu_i(X\setminus V'_{i,\epsilon,\delta})\le \frac{1}{1-\delta}(1-\mu_i(V_{i,\epsilon}))\le \frac{1}{1-\delta}r_\epsilon.$$
		Hence by \eqref{e-3},
		\begin{align*}
			\mu_i(V_{i,\epsilon,\delta})\ge \mu_i(W_{i,\epsilon,\delta})-\mu_i(X\setminus V'_{i,\epsilon,\delta})\ge  \mu_i(W_{i,\epsilon,\delta})-\frac{1}{1-\delta}r_\epsilon.
		\end{align*}
		Since $r_\epsilon\to 0$ as $\epsilon\to0$, one has  by \eqref{eq-11111} that
		\begin{align}\label{eq-1}\lim_{\delta\to 0}\liminf_{\epsilon\to 0}\mu_i(V_{i,\epsilon,\delta})=1.\end{align}

		Now we show that $V_{1,\epsilon,\delta}^{k_1}\times  V_{2,\epsilon,\delta}^{k_2}\times \cdots\times  V_{I,\epsilon,\delta}^{k_I}\subset W_\epsilon$. Fix a $$\overline{x}=(x_{1,1},\cdots,x_{1,k_1},\cdots,x_{I,k_I})\in V_{1,\epsilon,\delta}^{k_1}\times  V_{2,\epsilon,\delta}^{k_2}\times \cdots\times  V_{I,\epsilon,\delta}^{k_I}.$$ For $i\in\{1,2,\cdots,I\}$ with $k_i\ge2$ and $k\in\{1,2,\cdots,k_i\}$, by \eqref{e-111}, \eqref{e-3} and \eqref{e-1}, one has that $ V_{i,\epsilon}\cap \text{supp}\mu_{i,x_{i,k}}$ contains at least $k_i$ points such that  the distance between them are larger than $2\epsilon$. By Lemma \ref{lem-1}, for $i\in\{1,2,\cdots,I\}$ with $k_i\ge2$ one can choose $y_{i,k}\in  V_{i,\epsilon}\cap \text{supp}\mu_{i,x_{i,k}}$  for $1\le k\le k_i$ such that  the distance between them are larger than $\epsilon$. If $k_i=1$, by \eqref{e-8}, there is $y_{i,k_i} \in V_{i,\epsilon}\cap \text{supp}\mu_{i,x_{i,k_i}}$. Combining with \eqref{e-30}, the distance between $y_{1,1},\cdots,y_{1,k_1},\cdots,y_{I,k_I}$ are larger than $\epsilon$. This implies $(y_{1,1},\cdots,y_{1,k_1},\cdots,y_{I,k_I})\in \text{supp}\overline{\mu}_{\overline{x}}\cap (X^{\sum_{i=1}^Ik_i}\setminus X^{\sum_{i=1}^Ik_i}_2(\epsilon))$. Since $(X^{\sum_{i=1}^Ik_i}\setminus X^{\sum_{i=1}^Ik_i}_2(\epsilon))$ is open, one has $\overline{\mu}_{\overline{x}}((X^{\sum_{i=1}^Ik_i}\setminus X^{\sum_{i=1}^Ik_i}_2(\epsilon)))>0$ and then $\overline{x}\in W_\epsilon$. Therefore  $V_{1,\epsilon,\delta}^{k_1}\times  V_{2,\epsilon,\delta}^{k_2}\times \cdots\times  V_{I,\epsilon,\delta}^{k_I}\subset W_\epsilon$.

		By  \eqref{e-9} and \eqref{eq-1}, we can fix $\epsilon_n\to 0,\delta_n\to 0$ as $n\to\infty$ such that $$\lim_{n\to\infty}\mu_i(V_{i,\epsilon_n,\delta_n})=1\text{ for }i\in\{1,2,\cdots,I\}.$$
		Then
		\begin{align*}&\lim_{\epsilon\to0}\inf_{\tau\in\mathcal{E}}\tau_{\mathcal{K}}(W_{\epsilon})\\
			&\ge \liminf_{n\to\infty}\inf_{\tau\in\mathcal{E}}\tau_{\mathcal{K}}((V_{1,\epsilon_n,\delta_n})^{k_1}\times(V_{2,\epsilon_n,\delta_n})^{k_2}\times\cdots\times(V_{I,\epsilon_n,\delta_n})^{k_I})\\
			&\ge \liminf_{n\to\infty}\inf_{\tau\in\mathcal{E}}\left(\tau_{\mathcal{K}}(\overline{X})-\sum_{j=1}^{\sum_{i=1}^Ik_i}\tau_{\mathcal{K}}(X^{j-1}\times (X\setminus V_{i(j),\epsilon_n,\delta_n})\times X^{\sum_{i=1}^Ik_i-j})\right)\\
			&=\liminf_{n\to\infty}\left(1-\sum_{j=1}^{\sum_{i=1}^Ik_i}\mu_{i(j)}( X\setminus V_{i(j),\epsilon_n,\delta_n})\right)\\
			&=1.
		\end{align*}
		This ends the proof of Claim \ref{cl-6}.
	\end{proof}
	\section{Proof of Theorem \ref{thm-h} and  Theorem \ref{thm-B}}
	In this section, we prove Theorem \ref{thm-h} and  Theorem \ref{thm-B}. The proof of Theorem \ref{thm-h} is similar to the one of \cite[Theorem 3.1]{HLSY2}.
	We collect some known properties about  almost finite to one extensions.
	Let $\pi : (X, G)\to  (H, G)$ be an extension of $G$-system Let $\pi^{-1} : H \to 2^X, y 
	\to \pi^{-1}(y)$.
	Then it is easy to verify that $\pi^{-1}$ is a upper semicontinuous map, and the set $H_c$ of
	continuous points of $\pi^{-1}$ is a dense $G_\delta$ subset of $H$. Let
	$$H' =\overline{\{\pi^{-1}(y) : y \in H_c\}},$$
	where the closure is taken in $2^X$. It is obvious that $H'\subset 2^X$. Note that for each
	$A \in H'$, there is some $y\in H$ such that $A\subset  \pi^{-1}(y)$, and hence $A\to y$ define a map
	$\tau:H'\to H$.  In \cite{V}, Veech shown that $(H',G)$ is a minimal topological dynamical system and $\tau:(H',G)\to(H,G)$ is an almost one to one factor map.  The following lemma is from \cite[Corollary 2.19]{HLSY2}.
	\begin{lem}\label{N-1}
		Let $\pi:(X, G)\to (H, G)$ be an extension with $(H, G)$ being minimal. If $\pi$
		is an almost $N$ to one extension for some $N\in\mathbb{N}$, then the cardinality of each element of
		$H'$ is $N$, where $(H',G)$ is the minimal system defined above.
	\end{lem}
	Now we prove  Theorem \ref{thm-h}.
	\begin{proof}[Proof of Theorem \ref{thm-h}]Let $\tau$ and $(H',G)$ be as above. Then $\tau$ is almost one to one and $|A|=K_2$ for $A\in  H'$. The two  maps $\tau^{-1}$ and $\pi^{-1}$ is upper semicontinuous and hence
		measurable. By Lusin's theorem, we may therefore choose a compact set $E_1$ of $H$ with $\nu(E_1)>0$ such that $\tau^{-1}|_{E_1}$ and $\pi^{-1}|_{E_1}$ are continuous and $|\pi^{-1}(h)|=K_1$ for $h\in E_1$.
		Let $E$ be the support of $\nu|_{E_1}$ and $h_0\in E$. Then $E$ is a compact subset of $E_1$ with $\nu(U\cap E)>0$ for all open neighborhood $U$ of $h_0$. 
		
		Now we fix $K\le \lceil\frac{K_1}{K_2}\rceil, K\in\mathbb{N},K\ge 2$ and $L=(K-1)K_2+1\le K_1$. Then there is distinct $x_1, x_2,\cdots,x_L\in\pi^{-1}(h_0)$. Since $(X,G)$ is minimal, $\cup_{A\in \tau^{-1}(h_0)}A=\pi^{-1}(h_0)$.
		We can find $L'\in\mathbb{N}$ and distinct $A_1,A_2,\cdots,A_{L'}\in \tau^{-1}(h_0)$ such that  $\{x_1,x_2,\cdots,x_L\}\subset\bigcup_{l=1}^{L'}A_l$.  Then 
		$$L'\ge  \lceil\frac{L}{K_2}\rceil=K\ge 2.$$
		We have the following claim.
		\begin{cl}\label{claim-1} $(A_1,A_2,\cdots,A_{L'})\in \text{IT}_{L'}(H',G)$. 
		\end{cl}
		\begin{proof}[Proof of Claim \ref{claim-1}]Put
			$$r:=\min_{x,x'\in \bigcup_{l=1}^{L'}A_l\atop x\neq x'}d(x,x')>0.$$
			For each $l$ and $\epsilon\in(0,\frac{r}{3})$,
			if $A_l=\{x_1,x_2,\cdots,x_{K_2}\},$ then  define 
			$$U_l=\text{cl}(\{\{y_1,y_2,\cdots,y_{K_2}\}\in H':d(x_k, y_k)<\epsilon \text{ for }k\in\{1,2,\cdots,K_2\}\}).$$
			It is clear that  $U_l$ is  a closed neighborhood of $A_l$. To show  $(A_1,A_2,\cdots,A_{L'})\in \text{IT}_{L'}(H',G)$. 
			It is enough to show that the tuple $(U_1,U_2,\cdots ,U_{L'})$ has an infinite independence set.
			
			Since $\epsilon\in(0,\frac{r}{3})$, $U_l\cap U_{l'}=\emptyset$ for $1\le l<l'\le L'$. Since $\tau$ is almost one to one, $\text{cl}(\text{int}(\tau(U_l)))=\tau(U_l)$ for $1\le l\le L'$ and $\text{int}(\tau(U_l))\cap \text{int}(\tau(U_{l'}))=\emptyset$ for $1\le l<l'\le L'$.  By the continuity of $\tau^{-1}|_E$, there is a neighborhood $U$ of $h_0$ such that for $h\in U\cap E$, $\tau^{-1}(h)\cap U_l\neq\emptyset$ for $l\in\{1,2,\cdots,L'\}$. One has
			$\nu(\bigcap_{l=1}^{L'}\tau(U_l))\ge\nu(E\cap U)>0$.
			
			Let $E(H)$ be the Ellis semigroup of $(H,G)$. Then by Theorem \ref{thm-A1}, $(H, G )$ is a factor of $(E(H), G)$, where the factor map $\pi_e$ is given by
			$\pi_e: E(H)\to  H, t\to  th$ for some fixed $h\in H$. In particular, $\pi_e$ is open. Let $\tilde\nu$ be the Haar measure on $E(H)$. Then $\pi_e(\tilde{\nu})=\nu$. 
			
			$$\xymatrix{
				&~ &(X,G) \ar[d]_{\pi} & (H',G)  \ar[dl]_{\tau}              \\
				&(E(H),G,\tilde\nu )\ar[r]^{\pi_e} & (H,G,\nu) &}
			$$	
			
			Put $V_l=\pi_e^{-1}(\tau( U_l))$ for $l\in\{1,2,\cdots,L'\}$. Since $\pi_e$ is open, one has $\text{cl}(\text{int}(V_l))=V_l$  for $1\le l\le L'$, $\text{int}(V_l)\cap \text{int}(V_{l'})=\emptyset$ for $1\le l<l'\le L'$ and
			$\tilde\nu(\bigcap_{l=1}^{L'}V_l)>0$. By Proposition \ref{p-eliss}, there exists an infinite set $\mathcal{A}=\{g_1,g_2,\cdots\}\subset G$ such that for all $a\in\{1,2,\cdots,L'\}^\mathbb{N}$,
			$ \bigcap_{n=1}^\infty g_n^{-1}\text{int}(V_{a_n})\neq\emptyset.$
			This implies for all $N\in\mathbb{N}$ and $a\in\{1,2,\cdots,L'\}^\mathbb{N}$, $\bigcap_{n=1}^N g_n^{-1}\text{int}(\tau(U_{a_n}))$ is a nonempty open subset of $H$. Since $\tau$ is almost one to one, there is a $h_{a,N}\in \bigcap_{n=1}^N g_n^{-1}\text{int}(\tau(U_{a_n}))$ such that $|\tau^{-1}(h_{a,N})|=1$. This implies $|\tau^{-1}(g_nh_{a,N})|=1$ for $n=1,2,\cdots,N$ and then $\tau^{-1}(g_nh_{a,N})\subset U_{a_n}$. Therefore, $\bigcap_{n=1}^N g_n^{-1}U_{a_n}\neq\emptyset$ for all $a\in\{1,2,\cdots,L'\}^N$ and $N\in\mathbb{N}$. Since all $U_l$ are compact, one has $\bigcap_{n=1}^\infty g_n^{-1}U_{a_n}\neq\emptyset$ for all $a\in\{1,2,\cdots,L'\}^\mathbb{N}$. This implies  the tuple $(U_1,U_2,\cdots ,U_{L'})$ has an infinite independence sets and then  $(A_1,A_2,\cdots,A_{L'})\in \text{IT}_{L'}(H',G)$. 
		\end{proof}
		Let $(X^{K_2} , G)$ be the product system and $M_{K_2}(X)=\{\{x_1,x_2,\cdots,x_{K_2}\}: x_i\in X, 1\le i\le K_2\}$. Define
		$$p:X^{K_2}\to M_{K_2}(X):(x_1,\cdots,x_{K_2})\to \{x_1,\cdots,x_{K_2}\}.$$
		By Claim \ref{claim-1}, $(A_1,A_2,\cdots,A_{L'})\in \text{IT}_{L'}(H',G)\subset \text{IT}_{L'}(M_{K_2}(X),G)$. By Proposition \ref{p-ind} (4), there is $(A_1',A_2',\cdots,A_{L'}')\in \text{IT}_{L'}(X^{K_2},G)$ with $p(A_1',A_2',\cdots,A_{L'}')=(A_1,A_2,\cdots,A_{L'}).$ Then $$\bigcup_{l=1}^{L'}p(A_l')=\bigcup_{l=1}^{L'}A_l\supset\{x_1,x_2,\cdots,x_L\}.$$ For $k\in\{1,2,\cdots,K_2\}$, let $p_k$ be the projection on the $k$-th coordinate. Then there is $k_0\in \{1,2,\cdots,K_2\}$ such that 
		$$|\{p_{k_0}(A_l'):1\le l\le L'\}|\ge\left\lceil\frac{|\bigcup_{l=1}^{L'} p(A_l')|}{K_2}\right\rceil\ge\left\lceil\frac{L}{K_2}\right\rceil= K.$$
		Let $z_1, z_2,\cdots,z_K$ be $K$ distinct points in $\{p_{k_0}(A_l'):1\le l\le L'\}$. Then  one has $(z_1,z_2,\cdots,z_K)\in \text{IT}_{K}(X,G)$. This ends the proof of Theorem \ref{thm-h}.
	\end{proof}
	Now we prove  Theorem \ref{thm-A}.
	\begin{proof}[Proof of Theorem \ref{thm-A}] Suppose that $(X, G)$ is a minimal $G$-system. Let $(H,G)$ be the maximal equicontinuous factor of $(X,G)$ and $\pi$ be the factor map and $\nu$ be the uniquely ergodic measure of $(H,G)$. Since $(H,\nu,G)$ is ergodic, there exists $K_1\in\mathbb{N}\cap \{\infty \}$ such that 
		$|\pi^{-1}(y)|=K_1$   for $\nu$-a.e. $y\in H$. Since  $(X,G)$
		is tame, $(X,G)$ has no non-trivial $2$-IT-tuple. Then by Theorem \ref{thm-h},  $\lceil\frac{K_1}{K}\rceil=1$.  Since $K=\min_{y\in H}|\pi^{-1}y|$, one has $K_1\ge K$, one has $K_1=K$. The proof of Theorem \ref{thm-A} is completed.
	\end{proof}
	
	Now we prove  Theorem \ref{thm-B}.
	\begin{proof}[Proof of Theorem \ref{thm-B}] Suppose that $(X, G)$ is a minimal $G$-system. Let $(H,G)$ be the maximal equicontinuous factor of $(X,G)$ and $\pi$ be the factor map. Assume that $(X,G)$ satisfies the following conditions.
		\begin{itemize}
			\item[(1)] $\pi$ is almost $K_2$ to one for some $K_2\in\mathbb{N}$;
			\item[(2)] $(X,G)$ has no $K$-IT-tuple  for some $K\in\mathbb{N},K\ge 2$.
		\end{itemize}
		Let $\nu$ be the Haar measure on $H$. Assume  $\mu_1,\cdots, \mu_{K_2(K-1)+1}$ are different ergodic measures of $(X,G)$. Then there exists disjoint measurable subsets $V_1,\cdots, V_{K_2(K-1)+1}$ of $X$ such that 
		$$\mu_i(V_i)\ge 1-\frac{1}{2(K_2(K-1)+1)}\text{ for }i=1,2,\cdots, K_2(K-1)+1.$$
		Then 
		$$\nu(\bigcap_{i=1}^{K_2(K-1)+1} \pi(V_i))\ge \frac{1}{2}.$$
		This implies $|\pi^{-1}(h)|\ge K_2(K-1)+1$ for all 	$h\in\bigcap_{i=1}^{K_2(K-1)+1} \pi(V_i)$. By ergodicity, $\nu(\{h\in H:|\pi^{-1}(h)|\ge K_2(K-1)+1\})=1$. Then by Theorem \ref{thm-h}, $\text{IT}_K(X,G)\setminus\bigtriangleup^{(K)}(X)\neq\emptyset$ which is impossible by assumption.  We finish the proof of Theorem \ref{thm-B}.
	\end{proof}
	
	\section*{Acknowledgements}
	We would like to thank  J. Moreira for bringing to our attention the paper of V. Bergelson \cite{Be}. L. Xu was partially supported by  NSFC grant (12031019, 12371197) and the USTC Research Funds of the Double First-Class Initiative. C. Liu and X. Wang  was partially supported by NSFC grant (12090012, 12090010).

	\begin{appendix}
		\section{Proof of Lemma \ref{l-2}} Let $(X,\mu,G)$ be a $G$-measure preserving system. Recall that
		$$\mathcal{I}_\mu(G)=\{B\in\mathcal{B}_X^\mu:gB=B,\text{ for all } g\in G\},$$
		and the Kronecker $\sigma$-algebra  of $(X,\mu,G)$ is defined by $$\mathcal{K}_\mu(G)=\{B\in\mathcal{B}_X^\mu:\{U_g1_B:g\in G\}\text{ is precompact in }L^2(\mu)\}.$$
		In fact, $\mathcal{K}_\mu(G)$ is the $\sigma$-algebra generated by all almost periodic function $f\in L^2(\mu)$, i.e., $\{U_gf:g\in G\}$ is precompact in $L^2(\mu)$. We have the following.
		
		\begin{lem}\label{l-10}Let $I\in\mathbb{N}$ and $(X_i,\mu_i, G),i=1,2$ be $G$-measure preserving systems. Then 
			$$\mathcal{I}_{\mu_1\times\mu_2}(G)\subset \mathcal{K}_{\mu_1}(G)\times\mathcal{K}_{\mu_2}(G).$$
		\end{lem}
		\begin{proof}This proof is similar to that of \cite[Theorem 16, pp. 53]{HK}. Let $F(x_1,x_2)\in L^2(\mu_1\times\mu_2)$ be $G$-invariant. Define a bounded operator $A:L^2(\mu_1)\to L^2(\mu_2)$ by setting
			$$Ah(x_2)=\int_{X_1}F(x_1,x_2)h(x_1)d\mu_1(x_1)\text{ for }h\in L^2(\mu_1).$$ The operator $A$ is a Hilbert-Schmidt operator and thus is compact. The operator $A^*A$ is a positive, self-adjoint, compact operator on $L^2(\mu_1)$. Therefore, there exists a decreasing sequence $(t_j)$ of positive numbers (either finite or infinite) tending to $0$ as $j\to\infty$ and
			an orthonormal sequence $(f_j)$ in $L^2(\mu_1)$ such that for every $j\in \mathbb{N} $,
			$$A^*Af_j=t_jf_j,$$ and the sublinear space $V_j=\text{span}\{f_j \}$ is orthonogonal  to $\text{ker}(A)$.
			Since $F(x_1,x_2)$ is $G$-invariant, we have for all $g\in G$ and $h\in L^2(\mu_1)$  that 
			\begin{align}\label{e-19-2}\begin{split}
					(U_gA(h))(x_2)&=\int_{X_1}F(x_1,gx_2)h(x_1)d\mu_1(x_1)\\
					&=\int_{X_1}F(gx_1,gx_2)h(gx_1)d\mu_1(x_1)\\
					&=\int_{X_1}F(x_1,x_2)h(gx_1)d\mu_1(x_1)\\
					&=A(U_gh)(x_2).
				\end{split}
			\end{align}
			Since each $U_g^*=U_{g^{-1}}$, one has
			$$U_gA^*A=A^*AU_g,\text{ for all } g\in G.$$
			Therefore, for each $j$, $V_j$ is $G$-invariant.
			This means that for any $f\in V_j$
			$$\{U_gf:g\in G\}\subset \{\tilde f\in V_j:\|\tilde f\|_2=\|f\|_2\}.$$
			Since all $V_j$  are  finite dimension linear spaces, $\{U_gf:g\in G\}$ is precompact in $L^2(\mu_1)$ and then $f\in L^2(X_1,\mathcal{K}_{\mu_1}(G),\mu_1)$. Now let $V$ be the spanning of all $V_j$. Notice that $V_j\bot V_{j'}$ if $j\neq j'$, one has $ V\subset L^2(X_1,\mathcal{K}_{\mu_1}(G),\mu_1)$ and  $A(V^\bot)=0$. Let $W=A(V)$. By 
			\eqref{e-19-2}, $W\subset L^2(X_2,\mathcal{K}_{\mu_2}(G),\mu_2)$. Let $P_V$, $P_W$ and $P_{V\otimes W}$ be the projection on $V$, $W$ and $V\otimes W$ respectively. Then for all $h_1\in L^2(\mu_1), h_2\in L^2(\mu_2)$, one has 
			\begin{align*}
				&\int_{X_1\times X_2}F(x_1,x_2)h_1(x_1)h_2(x_2)d\mu_1\times\mu_2(x_1,x_2)\\&=\int_{ X_2}(Ah_1)h_2d\mu_2\\
				&=\int_{ X_2}(AP_Vh_1)h_2d\mu_2\\
				&=\int_{ X_2}(AP_Vh_1)(P_Wh_2)d\mu_2\\
				&=\int_{X_1\times X_2}F(P_{V\otimes W}(h_1\otimes h_2))d\mu_1\times\mu_2\\
				&=\int_{X_1\times X_2}(P_{V\otimes W}F)(h_1\otimes h_2)d\mu_1\times\mu_2.
			\end{align*}
			This implies 
			$F=P_{V\otimes W}F$ and then $F$ is $\mathcal{K}_{\mu_1}(G)\times\mathcal{K}_{\mu_2}(G)$ measurable.
			Hence 
			$$\mathcal{I}_{\mu_1\times\mu_2}(G)\subset \mathcal{K}_{\mu_1}\times\mathcal{K}_{\mu_2}.$$
			This ends the proof of Lemma \ref{l-10}.
		\end{proof}Now we prove Lemma \ref{l-2}.
		\begin{proof}Let $I\in\mathbb{N}$ and $(X_i,\mu_i, G),i=1,2,\cdots,I$ be a $G$-measure preserving system. If $I=1$, it is clear true. Assume that $I\ge 2$. By Lemma \ref{l-10}, $$\mathcal{I}_{\mu_1\times\mu_2\times\cdots\times\mu_I}\subset \mathcal{B}_X^{\mu_1}\times\mathcal{B}_X^{\mu_2}\times\cdots\times\mathcal{B}_X^{\mu_{i-1}}\times\mathcal{K}_{\mu_i}(G)\times\mathcal{B}_X^{\mu_{i+1}}\times\mathcal{B}_X^{\mu_{i+2}}\times\cdots\times\mathcal{B}_X^{\mu_{I}},$$ for $1\le i\le I.$
			This implies $\mathcal{I}_{\mu_1\times\mu_2\times\cdots\times\mu_I}(G)\subset \mathcal{K}_{\mu_1}(G)\times\mathcal{K}_{\mu_2}(G)\times\cdots\times\mathcal{K}_{\mu_I}(G).$ We finish the proof of Lemma \ref{l-2}. 
		\end{proof}
		\section{Proof of Theorem \ref{th-3}}  
		Firstly we show that $h^{*}_{\mu}(G,\alpha)= H_\mu(\alpha|\mathcal{K}_\mu)$ for a $G$-measure preserving system $(X,\mu,G)$ and $\alpha\in\mathcal{P}_X^\mu$, where $\mathcal{K}_\mu=\mathcal{K}_\mu(G)$ is the Kronecker $\sigma$-algebra of $(X,\mu,G)$. The following lemma is from \cite[pp. 94]{P}.
		\begin{lem} \label{lem-2}Let $(X,\mathcal{B},\mu)$ be a Borel probability space and $r\geq 1$ be a fixed integer. For each $\epsilon >0$, there exists  $\delta=\delta(\epsilon,r)>0$, such that if  $\alpha=\{A_1,A_2,\ldots,A_r\}$ and $\eta=\{B_1,B_2,\ldots,B_r\}$ are any two finite measurable partitions of $(X,\mathcal{B},\mu)$ with $\sum_{j=1}^r\mu(A_j\Delta B_j)< \delta$ then $H_\mu(\alpha|\eta)+H_\mu(\eta|\alpha)<\epsilon$.
		\end{lem}
		\medskip
		
		By Lemma \ref{lem-2}, we can get the following result.
		\begin{lem}\label{lem-11}
			Let $(X,\mu,G)$ be a $G$-measure preserving system and $B\in \mathcal{B}_X^\mu$. If $\{U_g1_B:g\in G\}$ is precompact in $L^2(\mu)$, then $h^{*}_{\mu}(G,\{B,B^c\})=0.$	
		\end{lem}
		\begin{proof}
			Let $\eta=\{B,B^c\}$. If $\{U_g1_B:g\in G\}$ is precompact in $L^2(\mu)$, then for any infinite sequence $\mathcal{A}=\{ g_i\}_{i\in\mathbb{N}}$ of $G$, $\{U_{g_i}1_B:i\in \mathbb{N}\}$ is precompact in $L^2(\mu)$.  So for $\epsilon >0$, there exists $s\in \mathbb{N}$ such that for any $i\in \mathbb{N}$ there is $j_i\in\{1,2,\cdots,s\}$ such that $$\mu(g_i^{-1}B^c\Delta g_{j_i}^{-1}B^c)=\mu(g_i^{-1}B\Delta g_{j_i}^{-1}B)=\|U_{g_i}1_B-U_{g_{j_i}}1_B\|^2_2<\frac \delta 2.$$ By Lemma \ref{lem-2}, one has  $$H_\mu(g_i^{-1}\eta|g_{j_i}^{-1}\eta)+H_\mu(g_{j_i}^{-1}\eta|g_i^{-1}\eta)<\epsilon\text{ for }i\in\mathbb{N}.$$
			Thus for any $ i>s$, we have  $$H_\mu(g_i^{-1}\eta|\bigvee_{j=1}^{i-1}g_{j}^{-1}\eta)\leq H_\mu(g_i^{-1}\eta|g_{j_i}^{-1}\eta)<\epsilon.$$
			Hence $$h^{\mathcal{A}}_{\mu}(G,\eta)=\limsup\limits_{n\to \infty}\frac{1}{n}\sum_{i=2}^{n}H_\mu(g_i^{-1}\eta|\bigvee_{j=1}^{i-1}g_{j}^{-1}\eta)\leq\epsilon.$$ Let $\epsilon \rightarrow 0.$ We obtain that $h^{\mathcal{A}}_{\mu}(G,\eta)=0.$ By the arbitrary of $\mathcal{A}$, we have  $h^{*}_{\mu}(G,\eta)=0.$ This ends the proof of Lemma \ref{lem-11}.
		\end{proof}
		\medskip
		
		\begin{thm} \label{lem-4}
			Let $(X,\mu,G)$ be a $G$-measure preserving system, and $\alpha\in\mathcal{P}_X^\mu$. Then  $h^{*}_{\mu}(G,\alpha)=H_\mu(\alpha|\mathcal{K}_\mu)$.
			
		\end{thm}
		\begin{proof}We firstly, show that 	$h^{*}_{\mu}(G,\alpha)\leq H_\mu(\alpha|\mathcal{K}_\mu).$ Let $\mathcal{A}=\{g_1,g_2,\cdots\}$ be a sequence of $G$.
			Since $(X,\mathcal{B}_X^\mu, \mu)$ is separable, there exist countably many finite measurable partitions  $\{\eta_k\}_{k\in\mathbb{N}}\subset \mathcal{K}_\mu $, such that $\lim\limits_{k \rightarrow \infty }H_{\mu}(\alpha|\eta_k)=H_\mu(\alpha|\mathcal{K}_\mu)$. Hence for any $k\in\mathbb{N}$, we have
			\begin{align*}
				h^\mathcal{A}_{\mu}(G,\alpha)
				=\limsup\limits_{l\to \infty}\frac{1}{l}H_\mu\big(\bigvee_{i=1}^lg_i^{-1}\alpha\big)
				\leq \limsup\limits_{l\to \infty}{\frac{1}{l}H_\mu(\bigvee_{i=1}^lg_i^{-1}(\alpha \vee \eta_k))}
			\end{align*}
			Since $\eta_k\subset \mathcal{K}_\mu$ for all $k\in \mathbb{N}$, one has by Lemma \ref{lem-11} that 
			$$h^{\mathcal{A}}_{\mu}(G,\eta_k)=\lim\limits_{l\to \infty}\frac{1}{l}H_\mu(\bigvee_{i=1}^lg_i^{-1}\eta_k)=0.$$
			Hence
			\begin{align*}
				h^\mathcal{A}_{\mu}(G,\alpha)&\leq \limsup\limits_{l\to \infty}{\frac{1}{l}H_\mu(\bigvee_{i=1}^lg_i^{-1}(\alpha \vee \eta_k))}-\lim\limits_{l\to \infty}\frac{1}{l}H_\mu(\bigvee_{i=1}^lg_i^{-1}\eta_k)\\
				&=\limsup\limits_{l\to \infty}{\frac{1}{l}H_\mu\big(\bigvee_{i=1}^lg_i^{-1}\alpha|\bigvee_{i=1}^lg_i^{-1}\eta_k\big)}\\
				&\leq \limsup\limits_{l\to \infty}{\frac{1}{l}\sum_{i=1}^lH_\mu(g_i^{-1}\alpha|g_i^{-1}\eta_k)}\\
				&=H_\mu(\alpha|\eta_k).
			\end{align*}
			Finally, let $k\rightarrow\infty$. We get
			$h^{\mathcal{A}}_{\mu}(G,\alpha)\leq H_\mu(\alpha|\mathcal{K}_\mu).$
			By the arbitrary of $\mathcal{A}$, we have  
			$h^{*}_{\mu}(G,\alpha)\leq H_\mu(\alpha|\mathcal{K}_\mu).$
			
			Now we prove the opposite side. 
			Given any $\beta\in\mathcal{P}_X^\mu$, write
			$\alpha=\{A_1, A_2, \ldots, A_l\}$ and $\beta=\{B_1, B_2, \ldots,
			B_t\}$.
			By Theorem \ref{thm:weak mixing}, $1_{A_i}-\mathbb{E}(1_{A_i}|\mathcal{K}_{\mu}) \in\mathcal{H}_{wm}$ for $1\le i\le l$.  There exists a sequence 
			$\mathcal{A}=\{g_i\}_{i=1}^\infty$ of $G$ such that
			\begin{equation}\label{eq:3-2}
				\lim_{i\rightarrow \infty}\langle
				U_{g_i}(1_{A_k}-\mathbb{E}(1_{A_k}|\mathcal{K}_{\mu})),1_{B_j}\rangle=0
			\end{equation}
			for any $1\leq k \leq l$ and $1\leq j\leq t$. Hence
			\begin{align*}
				&\liminf_{i\rightarrow \infty}{H_\mu(g_i^{-1}\alpha|\beta)}\\
				=&\liminf_{i\rightarrow \infty}{\sum_{k,j}-\mu\left(g_i^{-1}A_k\cap B_j\right)}
				\log{\left(\frac{\mu(g_i^{-1}A_k\cap B_j)}{\mu(B_j)}\right)}\\
				=&\liminf_{i\rightarrow \infty}{\sum_{k,j}-\langle U_{g_i}1_{A_k},1_{B_j}\rangle\log{\left(\frac{\langle U_{g_i}1_{A_k},1_{B_j}\rangle}{\mu\left(B_j\right)}\right)}}\\
				\overset{\eqref{eq:3-2}}=&\liminf_{i\rightarrow\infty}{\sum_{k,j}-\langle
					U_{g_i}\mathbb{E}\left(1_{A_k}|\mathcal{K}_{\mu}\right),1_{B_j}\rangle\log{\left(\frac{\langle
							U_{g_i}\mathbb{E}\left(1_{A_k}|\mathcal{K}_{\mu}\right),1_{B_j}\rangle}{\mu\left(B_j\right)}\right)}}.
			\end{align*}
			Let $$a_{kj}^i=-\langle
			U_{g_i}\mathbb{E}\left(1_{A_k}|\mathcal{K}_{\mu}\right),1_{B_j}\rangle
			\log{\left(\frac{\langle
					U_{g_i}\mathbb{E}\left(1_{A_k}|\mathcal{K}_{\mu}\right),1_{B_j}\rangle}{\mu\left(B_j\right)}\right)}
			\text{ and }\mu_{B_j}(\cdot)=\frac{\mu(\cdot\cap B_j)}{\mu(B_j)}.$$ By the
			concavity of $-x\log{x}$, we conclude that
			\begin{align*}
				\frac{a_{kj}^i}{\mu\left(B_j\right)}
				&=-\left(\int_{B_j}{\frac{U_{g_i}\mathbb{E}\left(1_{A_k}|\mathcal{K}_{\mu}\right)}
					{\mu\left(B_j\right)}}\,\mathrm{d}\mu\right)\log{\left(\int_{B_j}
					{\frac{U_{g_i}\mathbb{E}\left(1_{A_k}|\mathcal{K}_{\mu}\right)}{\mu\left(B_j\right)}}\,\mathrm{d}\mu\right)}\\
				&=-\left(\int_{B_j}{{U_{g_i}\mathbb{E}\left(1_{A_k}|\mathcal{K}_{\mu}\right)}}\,\mathrm{d}\mu_{B_j}\right)
				\log{\left(\int_{B_j}{{U_{g_i}\mathbb{E}\left(1_{A_k}|\mathcal{K}_{\mu}\right)}}\,\mathrm{d}\mu_{B_j}\right)}\\
				&\geq -\int_{B_j}{{U_{g_i}\mathbb{E}\left(1_{A_k}|\mathcal{K}_{\mu}\right)}}
				\log{\left(U_{g_i}\mathbb{E}\left(1_{A_k}|\mathcal{K}_{\mu}\right)\right)}\,\mathrm{d}\mu_{B_j}\\
				&=-\int_{B_j}{\frac{U_{g_i}\mathbb{E}\left(1_{A_k}|\mathcal{K}_{\mu}\right)}{\mu\left(B_j\right)}}
				\log{\left(U_{g_i}\mathbb{E}\left(1_{A_k}|\mathcal{K}_{\mu}\right)\right)}\,\mathrm{d}\mu.
			\end{align*}
			Therefore, we have
			\begin{align*}
				\sum_{k,j}a_{kj}^i
				&\geq \sum_{k,j}-\int_{B_j}{U_{g_i}\mathbb{E}\left(1_{A_k}|\mathcal{K}_{\mu}\right)}
				\log{\left(U_{g_i}\mathbb{E}\left(1_{A_k}|\mathcal{K}_{\mu}\right)\right)}\,\mathrm{d}\mu\\
				&=\sum_k -\int_X{U_{g_i}\mathbb{E}\left(1_{A_k}|\mathcal{K}_{\mu}\right)}
				\log{\left(U_{g_i}\mathbb{E}\left(1_{A_k}|\mathcal{K}_{\mu}\right)\right)}\,\mathrm{d}\mu\\
				&=\sum_k -\int_X{\mathbb{E}\left(1_{A_k}|\mathcal{K}_{\mu}\right)}
				\log{\left(\mathbb{E}\left(1_{A_k}|\mathcal{K}_{\mu}\right)\right)}\,\mathrm{d}\mu\\
				&=H_{\mu}(\alpha|\mathcal{K}_{\mu}).
			\end{align*}
			This shows that
			\begin{equation}\label{eq:3-3}
				\liminf_{i\to \infty}{H_\mu\left(g_i^{-1}\alpha|\beta\right)} \geq
				H_\mu\left(\alpha|\mathcal{K}_\mu\right).
			\end{equation}
			
			By the above discussion, we can obtain inductively a sequence $\mathcal{A}'=\{h_i\}_{i=1}^{\infty}$ of $G$ such that for each $i \in
			\mathbb{N}$, one has
			$$H_\mu\left(h^{-1}_i\alpha|\bigvee_{j=1}^{i-1} h^{-1}_j\alpha\right)\geq H_\mu(\alpha|\mathcal{K}_\mu)-\frac{1}{2^i}.$$
			Thus,
			\begin{align*}
				H_{\mu}\left(\bigvee_{i=1}^nh_i^{-1}\alpha\right) 
				&=
				H_{\mu}(h_{1}^{-1}\alpha)+H_{\mu}(h_{2}^{-1}\alpha|h_{1}^{-1}\alpha)+\ldots+
				H_{\mu}\left(h_{n}^{-1}\alpha|\bigvee_{j=1}^{n-1}h_{j}^{-1}\alpha\right)\\
				&\geq
				\sum_{i=1}^{n}\left(H_{\mu}(\alpha|\mathcal{K}_{\mu})-\frac{1}{2^{i}}\right)
				\geq n \cdot H_{\mu}(\alpha|\mathcal{K}_{\mu})-1.
			\end{align*}
			Therefore, we have
			\begin{align*}
				h^{\mathcal{A}'}_{\mu}(G,\alpha) &=\limsup_{n \rightarrow \infty}\frac{1}{n}H_{\mu}\left(\bigvee_{i=1}^nh_i^{-1}\alpha\right)\\
				&\geq \limsup_{n\rightarrow\infty}\frac{n
					H_{\mu}(\alpha|\mathcal{K}_{\mu})-1}{n}=H_{\mu}(\alpha|\mathcal{K}_\mu).
			\end{align*}
			This implies $	h^{*}_{\mu}(G,\alpha)\ge 	h^{\mathcal{A}'}_{\mu}(G,\alpha) \ge H_{\mu}(\alpha|\mathcal{K}_\mu)$. 	
			
			Therefore $	h^{*}_{\mu}(G,\alpha)=H_{\mu}(\alpha|\mathcal{K}_\mu)$.  This ends the proof of Theorem \ref{lem-4}.
		\end{proof}

		We need the following lemma from \cite[Lemma 3 in \S4 No. 2]{R}.
		\begin{lem}\label{ll-1}Let $\mu$ be a probability measure on $X$ and $\mathcal{F}$ be a sub-$\sigma$-algebra of $\B_{X}^\mu$.
			Let $\mu=\int \mu_yd\mu(y)$ be the disintegration of $\mu$ with respect to $\mathcal{F}$. Suppose $\mu_y$ is non-atomic for $\mu$-a.e. $y\in X$. If $0\le r\le 1$ and $A\in \B_{X}^\mu$ with $\mu_y(A)\le r$ for $\mu$-a.e. $y\in X$, then there exists $A'\in \B_{X}^\mu$ such that $A\subset A'$ and $\mu_y(A')=r$ for $\mu$-a.e. $y\in X$.
		\end{lem}
		Now we prove Theorem \ref{th-3}.
		\begin{proof}[Proof of Theorem \ref{th-3}]	Let $(X,\mu,G)$ be an ergodic $G$-measure preserving system and $\mu=\int_X\mu_xd\mu(x)$ be the disintegration of $\mu$ with respect to $\mathcal{K}_\mu(G)$. 
			
			We prove 	$h^{*}_{\mu}(G)\leq \log k_\mu$ firstly. If $k_\mu=\infty$, this is clear true. Assume that $k_\mu\in\mathbb{N}$. Then for $\alpha\in\mathcal{P}_X^\mu$, $$|\{A\in\alpha:\mu_x(A)>0\}|\le k_\mu\text{ for }\mu\text{-a.e. }x\in X.$$
			Hence by the convex of $-x\log x$ one has 
			\begin{align*}
				H_\mu(\alpha|\mathcal{K}_\mu)&=\int_X\sum_{A\in\alpha}-\mathbb{E}_\mu(1_A|\mathcal{K}_\mu)(x)\log \mathbb{E}_\mu(1_A|\mathcal{K}_\mu)(x)d\mu(x)\\
				&=\int_X\sum_{A\in\alpha}-\mu_x(A)\log \mu_x(A)d\mu(x)\\
				&=\int_X\sum_{A\in\alpha\atop \mu_x(A)>0}-\mu_x(A)\log \mu_x(A)d\mu(x)\\
				&\le \int_X-(\sum_{A\in\alpha\atop \mu_x(A)>0}\mu_x(A))\log \frac{\sum_{A\in\alpha\atop \mu_x(A)>0}\mu_x(A)}{|\{A\in\alpha: \mu_x(A)>0\}|}d\mu(x)\\
				&\le \int_X\log k_\mu d\mu(x)=\log k_\mu.
			\end{align*}
			
			By Theorem \ref{lem-4}, we have $h^{*}_{\mu}(G,\alpha)\leq H_\mu(\alpha|\mathcal{K}_\mu)\leq \log k_\mu$. With the arbitrariness of $\alpha$, we have  
			$h^{*}_{\mu}(G)\leq\log k_\mu .$

			Now we prove $h^{*}_{\mu}(G)\geq k_\mu$.  If  $k_\mu=\infty$, since   $(X,\mu,G)$ is ergodic, $\mu_x$ is atomless  for $\mu$-a.e. $x\in X$.  Then by Lemma \ref{ll-1}, for $k\in\mathbb{N}$ there exists $\alpha\in\mathcal{P}_X^\mu$ depending on $k$ such that
			$$\mu_x(A)=\frac{1}{k}\text{ for }\mu\text{-a.e.} x\in X\text{ and }A\in\alpha.$$ Then  $h^*_\mu(G)\ge h_\mu^*(G,\alpha)= H_\mu(\alpha|\mathcal{K}_\mu)=\log k$. Letting $k\to\infty$, one has $h^*_\mu(G)=\infty$.
			
			If $k_\mu<\infty$, since $(X,\mu,G)$ is ergodic, $\mu_x(y)=\frac{1}{k_\mu}$ for $y\in\text{supp}\mu_x$ for $\mu$-a.e. $x\in X$. Since $x\to\mu_x$ is measurable, there is a measurable partition $\alpha=\{A_1,\cdots,A_{k_\mu}\}$ of $X$ such that 
			$$\mu_x(A)=\frac{1}{k_\mu}\text{ for }\mu\text{-a.e. } x\in X\text{ and }A\in\alpha.$$
			Then $h^*_\mu(G)\ge h_\mu^*(G,\alpha)= H_\mu(\alpha|\mathcal{K}_\mu)=\log k_\mu$.
			
			Therefore $h^*_\mu(G)=\log k_\mu$. This ends the proof of  Theorem \ref{th-3}.
		\end{proof}
	\end{appendix}


\begin{thebibliography}{99}
		\bibitem{A} J. Auslander. A group theoretic condition in topological dynamics. Proceedings of the 18th Summer
		Conference on Topology and its Applications. Topology Proc. 28 (2004), no. 2, 327-334.
		\bibitem{A1}J. Auslander. Minimal flows and their extensions. Elsevier Science Ltd, 1988.
		
		\bibitem{Be}	V. Bergelson.  Minimal idempotents and ergodic ramsey theory. Topics
		in dynamics and ergodic theory, 310 (2003), pp. 8-39.
		\bibitem{B} R. Bowen. Entropy for group endomorphisms and homogeneous spaces. Trans. Am. Math. Soc. 153 (1971)
		401-414.
		\bibitem{D} E. I. Dinaburg. The relation between topological entropy and metric entropy. Sov. Math. 11 (1970) 13-16.
		\bibitem{E0} D. Ellis, R. Ellis and M. Nerurkar. The topological dynamics of semigroup actions. Trans. Amer. Math. Soc.
		353(4) (2001), 1279-1320.
		\bibitem{El}
		R. Ellis. Lectures on Topological Dynamics. W. A. Benjamin, Inc. New York, 1969. xv+211 pp.
		
		\bibitem{EG} R. Ellis and W. Gottschalk. Homomorphisms of transformation groups. Trans. Amer. Math. Soc. 94,	1960, 258-271.
		\bibitem{FG} G. Fuhrmann, E. Glasner, T. J\"ager and C. Oertel. Irregular model sets and tame dynamics. Trans. Amer. Math. Soc. 374 (2021), no. 5, 3703-3734.
		\bibitem{Fu}H. Furstenberg and Y. Katznelson. Idempotents in compact semigroups and Ramsey theory. Israel J. Math.
		68 (1989), 257-270.
		\bibitem{G3}E. Glasner. On tame dynamical systems. Colloq. Math. 105 (2006), no. 2, 283-295.
		\bibitem{G1}E. Glasner. The structure of tame minimal dynamical systems for general groups. Invent. Math. 211 (2018), no. 1, 213-244.
		\bibitem{G2} E. Glasner and M. Megrelishvili. Circularly ordered dynamical systems. Monatsh. Math. 185 (2018), no. 3, 415-441.
		\bibitem{G} T. N. T. Goodman. Topological sequence entropy. Proc. London Math. Soc. (3) 29 (1974), 331-350.
		
		\bibitem{Hi}
		N. Hindman and D. Strauss. Algebra in the Stone-\v{C}ech
		compactification. Theory and applications. Second revised and extended edition [of MR1642231]. De Gruyter Textbook. Walter de Gruyter \& Co. Berlin, 2012. xviii+591 pp.
		
		\bibitem{HK}B.  Host and B. Kra. Nilpotent structures in ergodic theory. Mathematical Surveys and Monographs. 236. American Mathematical Society, Providence, RI, 2018. X+427 pp.
		
		\bibitem{HLSY}W. Huang, S. Li, S. Shao and X. Ye. Null systems and sequence entropy pairs. Ergodic Theory Dynam. Systems 23 (2003), no. 5, 1505-1523.
		
		\bibitem{HLSY2} W. Huang, Z. Lian, S. Shao and X. Ye. Minimal systems with finitely many ergodic measures. J. Funct. Anal. 280 (2021), no. 12, Paper No. 109000, 42 pp.
		\bibitem{HLY} W. Huang, P. Lu and X. Ye. 
		Measure-theoretical sensitivity and equicontinuity.
		Israel J. Math. 183 (2011), 233-283.
		
		\bibitem{HMY}W. Huang, A. Maass and X. Ye. Sequence entropy pairs and complexity pairs for a measure. Ann. Inst. Fourier (Grenoble) 54 (2004), no. 4, 1005-1028. 
		
		\bibitem{HY}W. Huang and X. Ye. Combinatorial lemmas and applications to dynamics. Adv. Math. 220 (2009), no.
		6, 1689-1716.
		
		
		
		\bibitem{H} P. Hulse. Sequence entropy and subsequence generators. J. London Math. Soc. (2) 26 (1982), 441-450.
		
		\bibitem{KH} D. Kerr and H. Li. Independence in topological and $C^*$-dynamics. Math. Ann. 338 (2007), no. 4, 869-926.
		\bibitem{K} A. K\"ohler, Enveloping semigroups for flows. Proc. Roy. Irish Acad. Sect. A 95 (1995), no. 2, 179-191.
		\bibitem{Ku} A. G. Kushnirenko. On metric invariants of entropy type. Russ. Math. Surv. 22 (1967), 53-61.
		
		\bibitem{LY}C. Liu and K. Yan. Sequence entropy for amenable group actions. Phys. Scr. 98 (2023), 125268.
		\bibitem{Q} J. Qiu, Independence and almost automorphy of higher order. Ergodic Theory Dynam. Systems 43 (2023), no. 4, 1363-1381. 
		\bibitem{R}V. A. Rokhlin. On the fundamental ideas of measure theory.
		Amer. Math. Soc.
		Translation  (1952). no. 71, 55 pp. 
		
		
		\bibitem{V} W. A. Veech. Point-distal systems. Am. J. Math. 92 (1970), 205-242.
		\bibitem{P}P. Walters. An Introduction to Ergodic Theory (Graduate Texts in Mathematic, 79). Springer-Verlag,
		New York, 1982.
		
		
		
	\end{thebibliography}
\end{document}